\documentclass[11pt,reqno,final]{amsart}

\usepackage[utf8]{inputenc}
\usepackage[marginpar=2cm]{geometry}
\usepackage{setspace}
\usepackage{indentfirst}
\usepackage{tgtermes}
\RequirePackage{lineno}
\usepackage{mathrsfs,mathtools,amssymb,stmaryrd,tikz-cd}
\usepackage{graphicx}
\usepackage{enumitem}
\setlist[enumerate,1]{label=(\roman*),ref=\roman*}

\theoremstyle{plain}
\newtheorem{introthm}{Theorem}

\newtheorem{theorem}{Theorem}[section]
\newtheorem{lemma}[theorem]{Lemma}
\newtheorem{proposition}[theorem]{Proposition}
\newtheorem{corollary}[theorem]{Corollary}
\newtheorem{claim}{Claim}

\theoremstyle{definition}
\newtheorem{definition}[theorem]{Definition}

\theoremstyle{remark}
\newtheorem{remark}[theorem]{Remark}

\usepackage[backend=bibtex, bibencoding=utf8, style=numeric,
url=false,doi=false,eprint=true,isbn=false,
hyperref=auto,backref=false]{biblatex}
\usepackage[pagebackref=false,hidelinks,unicode=true,bookmarks=true,
linktoc=page,pdfstartview={FitH},final=true]{hyperref}

\allowdisplaybreaks[3]
\binoppenalty=\maxdimen
\relpenalty=\maxdimen
\DeclareMathOperator{\id}{id}
\DeclareMathOperator{\pr}{pr}

\DeclareMathOperator{\Hom}{Hom}

\DeclareMathOperator{\End}{End}
\DeclareMathOperator{\Der}{Der}

\DeclareMathOperator{\hkr}{hkr}
\DeclareMathOperator{\poly}{poly}
\DeclareMathOperator{\CE}{CE}

\DeclareMathOperator{\At}{At}
\DeclareMathOperator{\td}{td}
\DeclareMathOperator{\Td}{Td}

\DeclareMathOperator{\str}{str}
\DeclareMathOperator{\Ber}{Ber}

\DeclareMathOperator{\coker}{coker}
\DeclareMathOperator{\image}{im}

\DeclareMathOperator{\DGMfdCat}{DG^{+}-Mfd}
\DeclareMathOperator{\cLinfCat}{Bun^{+}-cL_{\infty}}

\newcommand{\cA}{\mathcal{A}}
\newcommand{\cB}{\mathcal{B}}
\newcommand{\cC}{\mathcal{C}}
\newcommand{\cD}{\mathcal{D}}
\newcommand{\cE}{\mathcal{E}}
\newcommand{\cF}{\mathcal{F}}

\newcommand{\cL}{\mathcal{L}}
\newcommand{\cM}{\mathcal{M}}
\newcommand{\cN}{\mathcal{N}}
\newcommand{\cO}{\mathcal{O}}
\newcommand{\cQ}{\mathcal{Q}}
\newcommand{\cR}{\mathcal{R}}

\newcommand{\frakg}{\mathfrak{g}}

\newcommand{\ZZ}{\mathbb{Z}}
\newcommand{\RR}{\mathbb{R}}

\newcommand{\argument}{\mathord{\color{black!25}-}}

\newcommand{\degree}[1]{\abs{#1}}

\newcommand{\xto}[1]{\xrightarrow{#1}}
\newcommand{\abs}[1]{\left|#1\right|}
\newcommand{\liederivative}[1]{\cL_{#1}}

\newcommand{\sections}[1]{\Gamma(#1)}

\newcommand{\XX}{\mathfrak{X}}

\newcommand{\transpose}{^{\top}}
\newcommand{\tensor}{\otimes}

\newcommand{\smooth}[1]{C^{\infty}({#1})}
\newcommand{\tangent}[1]{T{#1}}
\newcommand{\cotangent}[1]{T^{\vee}{#1}}
\newcommand{\tangentp}[2]{T_{#1}{#2}}

\newcommand{\grHH}{\mathrm{HH}_{\scriptscriptstyle{\oplus}}}

\newcommand{\hochschild}{d_{\mathcal{H}}}

\newcommand{\pDpoly}{{\mathcal{D}}_{\poly}}

\newcommand{\Tpoly}{{_{\scriptscriptstyle{\oplus}}}{\mathcal{T}}_{\poly}}
\newcommand{\Dpoly}{{_{\scriptscriptstyle{\oplus}}}{\mathcal{D}}_{\poly}}

\newcommand{\fM}{\mathsf{R}}			
\newcommand{\fcM}{\mathcal{R}}

\title{Atiyah class of DG manifolds of positive amplitude}

\author{Seokbong Seol}
\address{School of Mathematics, Korea Institute for Advanced Study,  Seoul 02455, Republic of Korea}
\email{azuredream89@kias.re.kr}

\thanks{The author is supported by the KIAS Individual Grants MG090801 and MG090802 at Korea Institute for Advanced Study.}

\begin{document}

\begin{abstract}
We prove that the Atiyah and Todd classes of differential graded (DG) manifolds of positive amplitude are invariant under the weak equivalences in the sense of Behrend--Liao--Xu.
As an application, we show that Hochschild cohomology of the second kind of DG manifolds of positive amplitude is also invariant under weak equivalences as well.
\end{abstract}

\maketitle

\tableofcontents


\section*{Introduction}
This paper, which is a sequel to~\cite{MR5026679}, studies the Atiyah class and Todd class of differential graded (DG) manifolds of positive amplitude. 

Recall that a ($\ZZ$-)graded manifold $\cM$ is a generalisation of a smooth manifold in a way that its algebra of functions $\smooth{\cM}$ is endowed with graded commutative elements, and a DG manifold $(\cM,Q)$ is a graded manifold $\cM$ equipped with a homological vector field $Q$. 
DG manifolds appear in various fields of mathematics, especially in the context of mathematical physics~\cite{MR1432574,MR2819233}. Complex manifolds, regular foliations, and $L_\infty$ algebras give rise to standard examples of DG manifolds.

In this paper, we focus on DG manifolds of \emph{positive amplitude}.
According to~\cite{MR4735657}, the category of such DG manifolds is equivalent to the category of bundles of positively graded curved $L_{\infty}[1]$ algebras $(M,L,\lambda)$. 
Here, $M$ is a smooth manifold, and $L$ is a positively graded vector bundle carrying a fibre-wise curved $L_{\infty}[1]$ algebra structure $\lambda$.

As shown in~\cite{MR4735657}, this category---equipped with appropriate fibrations and weak equivalences (see Definitions~\ref{defn:Fib} and~\ref{defn:WE})---forms a category of fibrant objects in the sense of Brown~\cite{MR341469}, and contains the category of smooth manifolds as a full subcategory.
A key advantage of this setting is that it provides a well-defined notion of homotopy fibre product. This serves as a substitute for the ordinary fibre product, which may \emph{not} be well-defined in the category of smooth manifolds.
According to~\cite{arXiv:2303.11140}, the category obtained by formally inverting weak equivalences aligns with certain ideas in derived geometry in the $C^{\infty}$-setting, such as those found in~\cite{MR3121621, MR2641940}.

In the study of DG manifolds, the Atiyah class plays an important role. 
The Atiyah class was originally introduced by Atiyah~\cite{MR86359} as the obstruction to the existence of holomorphic connections on holomorphic vector bundles. 
Kapranov~\cite{MR1671737} later showed that the Atiyah class plays a central role in the construction of an $L_{\infty}[1]$ algebra arising from a Kähler manifold, providing a reformulation of the Rozansky--Witten invariant \cite{MR1481135, MR1671725, MR2661534}. 
Subsequently, the Atiyah class was extended to DG manifolds, generalising the classical constructions---see~\cite{MR3319134}; see also \cite{MR2608525}. Geometrically, the Atiyah class of a DG manifold $(\cM,Q)$ is the obstruction to the existence of an affine connection compatible with the homological vector field $Q$.

The Atiyah class of DG manifolds recovers several familiar cases. 
For example, given a complex manifold $X$, the Atiyah class of $(T^{0,1}_{X}[1], \overline{\partial})$ recovers the classical Atiyah class, and similarly, for a regular foliation $F$, the Atiyah class of $(T_{F}[1],d_{F})$ recovers the Atiyah--Molino class~\cite{MR0281224}. See~\cite{MR3877426,MR4665716}. 
Also, for a vector bundle $E\to M$ with a section $s\in \sections{E}$, the Atiyah class of $(E[-1],\iota_{s})$ vanishes if and only if $s$ intersects the zero section cleanly~\cite{MR5026679}.

The Todd class of a DG manifold is defined in terms of the Atiyah class. It plays a central role in the Kontsevich formality theorem and, moreover, in the Duflo--Kontsevich-type theorem for DG manifolds~\cite{MR3754617}, 
which recovers both the classical Duflo theorem~\cite{MR0444841} 
and a result of Kontsevich concerning Hochschild cohomology of complex manifolds~\cite{MR2062626}. 
In particular, when $X$ is a compact Kähler manifold, the Todd class of $(T^{0,1}_{X}[1], \overline{\partial})$ recovers the classical Todd class, a central notion in the formulation of the Hirzebruch--Riemann--Roch theorem. Also, the Todd class of $(\frakg[1],d_{\CE})$ recovers the Duflo element of a Lie algebra $\frakg$. See~\cite{arXiv:math/9812009, MR3754617}.

This paper is devoted to answering the following natural question: are the Atiyah class and Todd class of DG manifolds of positive amplitude invariant under weak equivalences?

\begin{introthm}[Theorem~\ref{thm:main4}, Theorem~\ref{thm:main5}] \label{thm:IntroA}
The Atiyah class and Todd class of DG manifolds of positive amplitude are invariant under weak equivalences.
\end{introthm}

Expanding the techniques used to prove Theorem~\ref{thm:IntroA}, we also establish an analogous result for Hochschild cohomology of the second kind~\cite{MR2931331, MR4584414} of DG manifolds of positive amplitude. 
The Hochschild cohomology of DG manifolds in this paper is defined using poly-differential operators as in~\cite{MR3754617, arXiv:2511.12399}.
To be more precise, note that the space of poly-differential operators on a DG manifold $(\cM,Q)$ carries two compatible differentials. One is the Hochschild differential, which maps $p$-differential operators to $(p+1)$-differential operators. The other is induced by the homological vector field $Q$, which increases the internal degree arising from the grading of $\cM$. 
Together, these differentials form a double complex, and the Hochschild cohomology $\grHH(\cM,Q)$ is defined as the cohomology of the associated total complex, using the \emph{direct sum} totalisation.

The Hochschild cohomology carries a natural Gerstenhaber algebra structure, consisting of the cup product and Gerstenhaber bracket. In the positive-amplitude setting, we obtain the following:
\begin{introthm}[Theorem~\ref{thm:maingrHH1}]\label{thm:IntroC}
Let $\Psi:(\cM,Q)\to(\cN,R)$ be a weak equivalence between DG manifolds of positive amplitude. Then $\grHH^\bullet(\cM,Q)$ and $\grHH^\bullet(\cN,R)$ are isomorphic as Gerstenhaber algebras.
\end{introthm}

The proof of Theorem~\ref{thm:IntroA} is based on the following theorem:
\begin{introthm}[Theorem~\ref{thm:main1}, Theorem~\ref{thm:main2}, Theorem~\ref{thm:main3}] \label{thm:IntroB}
Let $\Psi:(\cM,Q)\to (\cN,R)$ be a weak equivalence between DG manifolds of positive amplitude. Then, for all nonnegative integers $p,q$,
there exists an isomorphism between the cohomology groups of $(p,q)$-tensor fields of $(\cM,Q)$ and $(\cN,R)$:
\[H^{\bullet}(\sections{\cM;\cotangent{\cM}^{\tensor q}\tensor \tangent{\cM}^{\tensor p}},\liederivative{Q}) \cong H^{\bullet}(\sections{\cN;\cotangent{\cN}^{\tensor q}\tensor \tangent{\cN}^{\tensor p}},\liederivative{R})\, .\]
\end{introthm}

When $(p,q)=(1,2)$, Theorem~\ref{thm:IntroB} implies that the cohomology groups in which the Atiyah classes of $(\cM,Q)$ and $(\cN,R)$ live are isomorphic. Appropriate choices of affine connections on $\cM$ and $\cN$ then induce Atiyah cocycles of $(\cM,Q)$ and $(\cN,R)$, whose cohomology classes agree under the isomorphism in Theorem~\ref{thm:IntroB}. The invariance of the Todd class under weak equivalences follows from the invariance of the Atiyah classes under weak equivalences.

Using Theorem~\ref{thm:IntroB}, one can show that the cohomology groups of poly-vector fields on $(\cM,Q)$ and $(\cN,R)$ are isomorphic as Gerstenhaber algebras. According to the Kontsevich--Duflo-type theorem for DG manifolds~\cite{MR3754617,arXiv:2608.07362}, Hochschild cohomology $\grHH^{\bullet}(\cM,Q)$ of $(\cM,Q)$ is isomorphic to the cohomology group of poly-vector fields on $(\cM,Q)$ as Gerstenhaber algebras. Combining these, we obtain Theorem~\ref{thm:IntroC}.

The proof of Theorem~\ref{thm:IntroB} relies on the case $(p,q)=(1,0)$, which corresponds to vector fields.
Given a weak equivalence $\Psi:(\cM,Q)\to (\cN,R)$, we compare their cochain complexes of vector fields via the diagram:
\[
\begin{tikzcd}
(\XX(\cM),\liederivative{Q}) \arrow{r}{\Psi_{\ast}} &(\sections{\Psi^{\ast}\tangent{\cN}},\liederivative{Q,R})& \arrow[swap]{l}{I} (\XX(\cN),\liederivative{R})
\end{tikzcd}
\]
where $\Psi^{\ast}\tangent{\cN}$ is the pullback bundle of $\tangent{\cN}$ along $\Psi$, and $\liederivative{Q,R}$ is a differential induced by $Q$ and $R$. 

The map $I$ can be shown to be a quasi-isomorphism using the fact that $(\cM,Q)$ is of positive amplitude. To show that $\Psi_{\ast}$ is a quasi-isomorphism, by Lemma~\ref{lem:ALF} (see also~\cite{MR4735657}), we may assume that the weak equivalence $\Psi$ is also a linear fibration. In this case, $\Psi_{\ast}$ is surjective, and the problem reduces to the acyclicity of the kernel $\ker \Psi_{\ast}$. By Lemma~\ref{lem:Locality}, it further reduces to a local statement: for each $p\in M$, there exists an open neighbourhood $U$ of $p$ such that $\ker \Psi_{\ast}$ restricted to $U$ is acyclic.

We distinguish between two types of points in $M$: classical points and non-classical points (see Definition~\ref{defn:WE}).
The case of a non-classical point $p\in M$ follows from the acyclicity of the DG algebra of smooth functions around $p$ (Lemma~\ref{lem:NonClassical}). 
For a classical point $p\in M$, 
the fibre of $\ker \Psi_{\ast}$ at $p$ is described by a finite exact sequence. Although the extension to an exact sequence in a neighbourhood of $p$ is not canonical, it can be constructed by a lower semi-continuity argument on the rank of vector bundle maps---see Lemma~\ref{lem:U-Exact}. 
This construction yields the local acyclicity of $\ker\Psi_{\ast}$ and thus proves the case $(p,q)=(1,0)$.

The case of differential 1-forms (i.e., $(p,q)=(0,1)$) is obtained using a similar argument by replacing the map $\Psi_{\ast}$ with the map
\[\Psi^{\star}:(\sections{\Psi^{\ast}\cotangent{\cN}},\liederivative{Q,R}) \to (\Omega^{1}(\cM),\liederivative{Q})\, .\]
The general case is obtained by taking iterated tensor products of vector fields and differential $1$-forms; in the positive-amplitude setting, these tensor products preserve quasi-isomorphisms.

We remark that an independent investigation of the same problem is being conducted by Chen, Liao, Xiang, and Xu~\cite{CLXX} using a different method.

\begingroup

\renewcommand{\addcontentsline}[3]{}

\section*{Notation}
Throughout this paper, we work over the field $\RR$ of real numbers.
Unless otherwise stated, all manifolds are assumed to be finite-dimensional, Hausdorff, and second countable.
We reserve symbols $M$ and $N$ for smooth manifolds and symbols $\cM$ and $\cN$ for graded manifolds with base $M$ and $N$, respectively. Given a graded vector bundle $E$, we denote by $E^{k}$ the degree $k$ component of $E$.

\endgroup

\section{DG manifolds of positive amplitude}
In this section, we review some general facts about differential graded (DG) manifolds of positive amplitude. For a detailed exposition, we refer to \cite{MR4735657}.

\subsection{DG manifolds of positive amplitude}
Let $M$ be a smooth manifold and let $\cO_{M}$ be the sheaf of smooth functions on $M$. A \textbf{graded manifold} $\cM$ is a pair $(M,\cA)$, where $M$ is a smooth manifold called the base of $\cM$, and $\cA$ is a sheaf of graded commutative $\cO_{M}$-algebras such that, for each sufficiently small open subset $U\subset M$, the restriction $\cA|_{U}$ is isomorphic to the sheafification of $\cO_{M}|_{U}\otimes_{\RR} S(V^{\vee})$. 
Here, $V$ is a fixed graded vector space independent of the choice of $U$, and $S(V^{\vee})$ denotes the symmetric algebra on $V^{\vee}$, viewed as the algebra of polynomial functions on $V$. 
We denote the graded algebra of global sections of $\cA$ by $\smooth{\cM}$.

A \textbf{DG manifold} is a graded manifold $\cM$ equipped with a homological vector field $Q\in \XX(\cM)$; 
that is, $Q$ is of degree $+1$ and satisfies $[Q,Q]=2Q^{2}=0$.

We say that a graded manifold $\cM$ is finite-dimensional if both $\dim M < \infty$ and $\dim V < \infty$ (viewing $V$ as an ordinary vector space). It is said to be of \textbf{positive amplitude} if the grading on $V$ is concentrated in strictly positive degrees---in this case, the grading of $S(V^{\vee})$ is concentrated in \emph{non-positive} degrees.
 
Throughout this paper, all graded manifolds are assumed to be finite-dimensional. 

\begin{remark}
In \cite{MR4735657}, a DG manifold of finite positive amplitude refers to a DG manifold whose graded manifold is both finite-dimensional and of positive amplitude. Since we assume that all graded manifolds are finite-dimensional, we simply call them DG manifolds of positive amplitude.
\end{remark}

Let $(\cM,Q)$ and $(\cN,R)$ be DG manifolds with bases $M$ and $N$, respectively, and
let $\cA$ and $\cB$ denote the sheaves associated with the graded manifolds $\cM$ and $\cN$, respectively. 
A morphism of graded manifolds $\Psi:\cM \to \cN$ consists of a morphism of smooth manifolds $f:M\to N$ and a morphism of sheaves of graded $\cO_{M}$-algebras $\psi:f^{\ast}\cB \to \cA$. 
A morphism $\Psi=(f,\psi):\cM\to \cN$ of graded manifolds induces a morphism of graded algebras $\Psi^{\ast}:\smooth{\cN}\to \smooth{\cM}$ which factors naturally through the induced map on the global sections $\psi|_{M}:\smooth{M}\tensor_{\smooth{N}}\smooth{\cN} \to \smooth{\cM}$. If $\Psi^{\ast}$ is compatible with homological vector fields $Q$ and $R$, we say $\Psi$ is a morphism of DG manifolds. 

DG manifolds of positive amplitude, together with their morphisms, form a category denoted by $\DGMfdCat$.

\subsubsection{Definitions and morphisms}
Let $L=\bigoplus_{k=a}^{b}L^{k} \to M$ be a graded vector bundle of finite rank, 
equipped with a sequence $\lambda=\{\lambda_{n}\}_{n\geq 0}$ of smooth bundle maps over $M$, where $\lambda_{n}:S^{n}L\to L$, and $S^{0}L=M\times \RR$ denotes the trivial line bundle over $M$. 
We may view $\lambda$ as a coderivation of the graded symmetric $\smooth{M}$-coalgebra $\sections{SL}$.

Assume that the fibres of $L$ are concentrated in positive degrees, i.e., $0<a\leq b$, and that $\lambda$ is of degree $+1$ and satisfies $\lambda^{2}=0$ as a coderivation. 
Then the triple $(M,L,\lambda)$ is called a \textbf{bundle of positively graded curved $L_{\infty}[1]$ algebras}. 
The condition $\lambda^{2}=0$ is equivalent to a sequence of compatibility conditions, starting with
\begin{equation}\label{eq:cL-infty}
 \lambda_{1}(\lambda_{0})=0, \quad \lambda_{2}(\lambda_{0},x)+\lambda_{1}^{2}(x)=0, \quad \cdots 
 \end{equation}
 for $x\in \sections{L}$.

A morphism of bundles of positively graded curved $L_{\infty}[1]$ algebras $\Phi:(M,L,\lambda)\to (N,E,\mu)$ 
consists of a morphism of smooth manifolds $f:M\to N$ and a sequence $\phi=\{\phi_{n}\}_{n\geq 1}$ of smooth bundle maps of degree $0$ over $M$, 
where $\phi_{n}:S^{n}L\to f^{\ast}E$, such that $\phi$ is compatible with $\lambda$ and $\mu$. The compatibility conditions include
\[
 \phi_{1}(\lambda_{0})= \mu_{0}, \quad 
 \phi_{2}(\lambda_{0},x)+\phi_{1}(\lambda_{1}(x))=\mu_{1}( \phi_{1}(x)), \quad \cdots
\]
for $x \in \sections{L}$. Here, by abuse of notation, we also denote the pullback map $f^{\ast}\mu_{n}:f^{\ast}S^{n}E \to f^{\ast}E$ by $\mu_{n}$.
If $\phi_{n}=0$ for all $n\geq 2$, then the morphism $\Phi$ is called linear. In this case, $\phi_{1}(\lambda_{0})=\mu_{0}$, and, for each $n\geq 1$, the remaining compatibility conditions reduce to
\[\phi_{1} \big( \lambda_{n}(x_{1},\ldots,x_{n})\big) = \mu_{n}(\phi_{1}(x_{1}),\ldots,\phi_{1}(x_{n}))\]
for $x_{1},\ldots, x_{n}\in \sections{L}$.

Bundles of positively graded curved $L_{\infty}[1]$ algebras and their morphisms form a category, which we denote by $\cLinfCat$.

Each $(M,L,\lambda)$ induces a DG manifold $(\cM,Q)$ of positive amplitude with base $M$. Its structure sheaf is the sheaf of sections of $SL^{\vee}$; in particular, $\smooth{\cM}=\sections{SL^{\vee}}$. 
The dual map
\[\lambda^{\transpose}=(\lambda_{0}^{\transpose},\lambda_{1}^{\transpose},\lambda_{2}^{\transpose},\ldots):L^{\vee}\to SL^{\vee}\] 
extends uniquely to a derivation $Q=\lambda^{\transpose}$ of $\sections{SL^{\vee}}$.
Note that the target of $\lambda^{\transpose}$ is the symmetric tensor bundle $SL^{\vee}$ (not its completion). This follows from the assumption that 
$L\to M$ is a graded vector bundle of finite rank concentrated in positive degrees and that $\lambda$ is of degree $+1$. 

The above assignment defines a functor from $\cLinfCat$ to $\DGMfdCat$. An analogue of Batchelor's theorem \cite{MR536951} holds for DG manifolds of positive amplitude.
\begin{proposition}[\cite{MR4735657}]\label{prop:CatBundle}
The category $\DGMfdCat$ is equivalent to the category $\cLinfCat$.
\end{proposition}

As a consequence of the above proposition, any DG manifold of positive amplitude can be identified with a bundle of positively graded curved $L_{\infty}[1]$ algebras.

\subsubsection{Categories of fibrant objects}
We recall the notion of a category of fibrant objects, introduced by Brown \cite{MR341469}. A category $\cC$ 
is called a \textbf{category of fibrant objects} if it has finite products, a terminal object $T$ (also called a final object), and is equipped with two distinguished classes of morphisms: fibrations and weak equivalences. A morphism that is both a fibration and a weak equivalence is called an acyclic fibration (or a trivial fibration). 
The fibrations and weak equivalences must satisfy the following:
\begin{enumerate}
\item ($2$-out-of-$3$) If two of the three morphisms $f:X\to Y$, $g:Y\to Z$ and $g\circ f:X\to Z$ are weak equivalences, then so is the third. Any isomorphism is a weak equivalence.
\item The composition of two fibrations is a fibration. Any isomorphism is a fibration.
\item The pullback of a fibration (resp.\  an acyclic fibration) along any morphism exists and is again a fibration (resp.\  an acyclic fibration).
\item For any object $X$, the morphism $X\to T$ is a fibration.
\item For any object $X$, there is an object $P_{X}$, called a path space for $X$, such that the diagonal map $\Delta_{X}:X\to X\times X$ factors through a weak equivalence $f:X\to P_{X}$ and a fibration $g:P_{X}\to X\times X$.
\end{enumerate}

The following lemma is known as the factorisation lemma.
\begin{lemma}[\cite{MR341469}] \label{lem:Fac}
Let $\cC$ be a category of fibrant objects. 
Any morphism $f:X\to Y$ admits a factorisation $f=p\circ i$, where $i$ is a weak equivalence and $p$ is a fibration. 
Moreover, the weak equivalence $i$ is a right inverse of an acyclic fibration.
\end{lemma}

\begin{corollary}\label{cor:Fac}
In a category of fibrant objects, the source and target of a weak equivalence are linked by a pair of acyclic fibrations.
\end{corollary}

\begin{theorem}[\cite{MR4735657}] \label{thm:CatFib}
The category $\DGMfdCat$ of DG manifolds of positive amplitude is a category of fibrant objects.
\end{theorem}
By Proposition~\ref{prop:CatBundle}, the notions of fibrations and weak equivalences in $\DGMfdCat$ correspond to
those in $\cLinfCat$.
Following~\cite{MR4735657}, we briefly describe fibrations and weak equivalences in $\cLinfCat$ below.

For the remainder of this section, let $\Phi=(f,\phi):(M,L,\lambda)\to (N,E,\mu)$ be a morphism of bundles of positively graded curved $L_{\infty}[1]$ algebras. 

\begin{definition}[\cite{MR4735657}]\label{defn:Fib}
We say $\Phi$ is a \textbf{fibration} if $f$ is a submersion and the linear component $\phi_{1}:L\to f^{\ast}E$ of $\phi$ is a degreewise surjective bundle map over $M$. If a fibration $\Phi$ is linear, then $\Phi$ is called a linear fibration.
\end{definition}

\begin{lemma}[\cite{MR4735657}] \label{lem:LinFib}
Every fibration can be factored as an isomorphism followed by a linear fibration.
\end{lemma}

We next recall the definition of a weak equivalence.

Let $L=\bigoplus_{k=1}^{b}L^{k}$. 
Recall that the zeroth coefficient of $\lambda$ is a section $\lambda_{0}:M\to L^{1}$, and the first coefficient of $\lambda$ is a bundle map $\lambda_{1}:L^{\bullet}\to L^{\bullet+1}$. For each point $p\in M$ satisfying $\lambda_{0}(p)=0$, 
one obtains a cochain complex
\begin{equation}\label{eq:TangentComplex}
\begin{tikzcd}
0 \arrow{r} & \tangentp{p}{M} \arrow{r}{D_{p}\lambda_{0}} & L^{1}|_{p} \arrow{r}{\lambda_{1}|_{p}} & L^{2}|_{p} \arrow{r}{\lambda_{1}|_{p}} & \cdots \arrow{r}{\lambda_{1}|_{p}} & L^{b}|_{p} \arrow{r} & 0
\end{tikzcd}
\end{equation}
called the tangent complex of $(M,L,\lambda)$ at $p$, where $L^{k}|_{p}$ denotes the fibre of $L^{k}$ at $p$ and
$D_{p}\lambda_{0}:\tangentp{p}{M}\to L^{1}|_{p}$ is defined by the composition
\[
\begin{tikzcd}
 \tangentp{p}{M} \arrow{r}{(\lambda_{0})_{\ast}|_{p}} & \tangentp{\lambda_{0}(p)}{L^{1}} \cong \tangentp{p}{M}\oplus L^{1}|_{p} \arrow{r}{\pr} & L^{1}|_{p}
\end{tikzcd}
\]
of the tangent map $(\lambda_{0})_{\ast}|_{p}$ of $\lambda_{0}$ at $p$ and the natural projection $\pr$. 
Since $\lambda_{0}(p)=0\in L^{1}|_{p}$, the isomorphism $\tangentp{\lambda_{0}(p)}{L^{1}} \cong \tangentp{p}{M}\oplus L^{1}|_{p}$ is canonical.

\begin{definition}[\cite{MR4735657}]\label{defn:WE}
A point $p\in M$ satisfying $\lambda_{0}(p)=0$ is called a \textbf{classical point} of $(M,L,\lambda)$. The set of all classical points
is called the classical locus and is denoted by $Z(\lambda_{0})$.
A morphism $\Phi=(f,\phi)$ is called a \textbf{weak equivalence} if
\begin{enumerate}
\item $f$ restricts to a bijection between the classical loci, and
\item for each $p\in Z(\lambda_{0})$, the induced maps 
\[f_{\ast}|_{p}:\tangentp{p}{M} \to \tangentp{f(p)}{N}, \quad \phi_{1}|_{p}:L|_{p}\to E|_{f(p)}\]
together define a quasi-isomorphism between the tangent complexes at $p$ and $f(p)$.
\end{enumerate}
\end{definition}
 
The following proposition shows that weak equivalences induce quasi-isomorphisms of DG algebras of functions.
 \begin{proposition}[\cite{MR5008683}]\label{prop:WEQI}
Let $\Psi:(\cM,Q)\to (\cN,R)$ be a weak equivalence between DG manifolds of positive amplitude. Then the induced morphism of DG algebras
\[\Psi^{\ast}:(\smooth{\cN},R)\to (\smooth{\cM},Q)\]
 is a quasi-isomorphism.
\end{proposition}

\subsection{Vector fields and differential 1-forms}
\label{sec:VFDF}

A \textbf{DG vector bundle} over a DG manifold is a vector bundle object in the category of DG manifolds. 
Given a DG manifold $(\cM,Q)$, 
it was shown in \cite{MR3319134} that a graded vector bundle $\cE$ over $\cM$ has a DG vector bundle structure over $(\cM,Q)$ if and only if its space of sections $\sections{\cE}$ admits a DG $(\smooth{\cM},Q)$-module structure.

\begin{remark}\label{rem:Sheaf}
It is often useful to describe DG vector bundles in terms of sheaves: if $(\cM,Q)$ is a DG manifold with base $M$, then a DG vector bundle is equivalent to a locally free sheaf of DG $(\cA,Q)$-modules on $M$, where $\cA$ denotes the sheaf associated with the graded manifold $\cM$. See \cite{MR3319134, MR2709144} for details.
\end{remark}

Given a DG manifold of positive amplitude, we describe the tangent and cotangent DG vector bundles in terms of their sections, namely vector fields and differential $1$-forms.

\subsubsection{Vector fields}
Let $(\cM,Q)$ be a DG manifold, and let $\Der(\smooth{\cM})$ denote the graded $\smooth{\cM}$-module of derivations of $\smooth{\cM}$.
The \textbf{tangent bundle} of $(\cM,Q)$ is defined by the DG $(\smooth{\cM},Q)$-module whose underlying graded module is $\sections{\tangent{\cM}}:=\Der(\smooth{\cM})$, equipped with differential $\liederivative{Q}=[Q,\argument]$, given by the graded commutator with $Q$. 
The graded module $\sections{\tangent{\cM}}$ is often denoted by $\XX(\cM)$, and its elements are called vector fields on $\cM$.

Assume that $(\cM,Q)$ is of positive amplitude. By Proposition~\ref{prop:CatBundle}, we may identify $(\cM,Q)$ with a bundle of positively graded curved $L_{\infty}[1]$ algebras $(M,L,\lambda)$. That is, we have $\smooth{\cM}=\sections{SL^{\vee}}$, and $Q=\lambda^{\transpose}$. Under this identification, we describe the DG module of vector fields $(\XX(\cM),\liederivative{Q})=(\Der(\sections{SL^{\vee}}),\liederivative{\lambda^{\transpose}})$ in terms of $(M,L,\lambda)$.

The bundle projection $\pi:L\to M$ is a morphism of graded manifolds, 
hence induces an exact sequence of graded vector bundles over the graded manifold $L$:
\[\begin{tikzcd}
0\arrow{r}& \pi^{\ast}L \arrow{r}& \tangent{L} \arrow{r}{\pi_{\ast}}& \pi^{\ast}\tangent{M} \arrow{r} & 0 \, ,
\end{tikzcd}\]
where $\pi_{\ast}$ is the tangent map of $\pi$. 
This induces an exact sequence of graded $\sections{SL^{\vee}}$-modules
\begin{equation}\label{eq:TangentSES}
\begin{tikzcd}
0 \arrow{r}& \sections{SL^{\vee}\tensor L} \arrow{r}& \Der(\sections{SL^{\vee}}) \arrow{r}{\pi_{\ast}} & \sections{SL^{\vee}\tensor \tangent{M}} \arrow{r}& 0 .
\end{tikzcd}
\end{equation}
Each choice of $\tangent{M}$-connection $\nabla$ on $L$ yields a splitting of the short exact sequence~\eqref{eq:TangentSES}, inducing
an isomorphism of graded $\sections{SL^{\vee}}$-modules
\begin{equation}\label{eq:SLTL}
\sections{SL^{\vee}\tensor ( \tangent{M}\oplus L)} \cong \Der(\sections{SL^{\vee}}).
\end{equation}
Indeed, 
denoting by $\widetilde{\nabla}$ the $\tangent{M}$-connection on $SL^{\vee}$ induced by $\nabla$,
the above isomorphism is characterised by
\begin{equation}\label{eq:VFBase}
\begin{gathered} 
X\mapsto \widetilde{\nabla}_{X} \in \Der(\sections{SL^{\vee}}), \qquad \forall X\in \sections{M;\tangent{M}}\\
s\mapsto \iota_{s} \in \Der(\sections{SL^{\vee}}), \qquad \forall s\in \sections{M;L}
\end{gathered}
\end{equation}
where the symbol $\iota_{s}$ denotes the interior product with $s\in \sections{L}$. 

Next, we describe the differential $\liederivative{\lambda^{\transpose}}:\Der(\sections{SL^{\vee}})\to \Der(\sections{SL^{\vee}})$ under the identification~\eqref{eq:SLTL}.
Let $L=\bigoplus_{k=1}^{b}L^{k}$ and $L^{0}:= \tangent{M}$. 
By Theorem~\ref{thm:BaseStr}, the DG $(\sections{SL^{\vee}},\lambda^{\transpose})$-module structure on 
$\sections{SL^{\vee}\tensor(\tangent{M}\oplus L)}$ 
is uniquely determined by maps
\[
d_{ij}|_{\sections{L^{j}}}: \sections{L^{j}}\to \sections{SL^{\vee}\tensor L^{i}}
\]
for $i>j\geq 0$.

By abuse of notation, let $\nabla:\sections{\tangent{M}}\times \sections{SL}\to \sections{SL}$ be the $\tangent{M}$-connection on $SL$ induced by the $\tangent{M}$-connection $\nabla$ on $L$. 
We may view the connection $\widetilde{\nabla}$ appearing in~\eqref{eq:VFBase} as the dual connection of $\nabla$, but restricted to the uncompleted dual bundle $SL^{\vee}$.

Regarding the bundle $SL^{\vee}\tensor L^{i}$ as a subbundle of $\Hom(SL, L^{i})$, we obtain the following explicit formulae.
For $i>0$ and $j>k>0$, the differential $\liederivative{\lambda^{\transpose}}$ is uniquely determined by
\begin{gather*}
d_{i0}(X)=\pr_{i}\circ [\lambda, \nabla_{X}] \qquad \forall X\in \XX(M),\\
d_{jk}(s)=(-1)^{k+1}\pr_{j}\circ (\iota_{s}\lambda) \qquad \forall s\in \sections{L^{k}},
\end{gather*}
where $\pr_{i}:\sections{SL}\to \sections{L^{i}}$ is the natural projection, and $\lambda$ is viewed as a $\smooth{M}$-coderivation $\lambda:\sections{SL}\to \sections{SL}$. The bracket $[\argument,\argument]$ denotes the graded commutator.

Altogether, we have proved the following.

\begin{proposition}\label{prop:Tangent}
Let $(\cM,Q)$ be a DG manifold of positive amplitude, and let $(M,L,\lambda)$ be the corresponding bundle of positively graded curved $L_{\infty}[1]$ algebras. 
Upon a choice of $\tangent{M}$-connection $\nabla$ on $L$, there is an isomorphism of DG modules over
$(\smooth{\cM},Q)=(\sections{SL^{\vee}},\lambda^{\transpose})$:
\[(\XX(\cM),\liederivative{Q}) \cong (\sections{SL^{\vee}\tensor ( \tangent{M}\oplus L)}, D_{\lambda})\, ,\]
where $D_{\lambda}=(d_{ij})_{0\leq i,j\leq b}$ is the differential whose components
\[d_{ij}:\sections{SL^{\vee}\tensor L^{j}}\to \sections{SL^{\vee}\tensor L^{i}}\]
(where $L^{0}:=\tangent{M}$) are given as follows:
\begin{enumerate}
\item $d_{ij}=0$ \qquad  if $i<j$,
\item $d_{ii}=Q\tensor \id$ \qquad  for all $i$,
\item $d_{i0}(X) = \pr_{i}\circ [\lambda,\nabla_{X}]$  \qquad for all $i>0$ and for $X\in \XX(M)$, 	\label{item:3}
\item $d_{ij}(s) = (-1)^{j+1} \pr_{i}\circ (\iota_{s}\lambda)$ \qquad for $i>j>0$ and for $s\in \sections{L^{j}}$.	\label{item:4}
\end{enumerate}
Here, $\pr_{i}:\sections{SL}\to \sections{L^{i}}$ denotes the natural projection, $\lambda$ is viewed as a $\smooth{M}$-coderivation $\lambda:\sections{SL}\to \sections{SL}$, and $\iota_{s}\lambda:\sections{SL}\to \sections{SL}$ denotes the interior product of $\lambda$ with $s$.
\end{proposition}

\subsubsection{Differential $1$-forms}
Similarly, we consider the cotangent bundle of a DG manifold.
Given a DG manifold $(\cM,Q)$,
the \textbf{cotangent bundle} of $(\cM,Q)$ is defined by the $\smooth{\cM}$-linear dual of $\XX(\cM)$, equipped with the induced differential. 
The corresponding DG $(\smooth{\cM},Q)$-module is denoted by $(\sections{\cotangent{\cM}},\liederivative{Q})$.
The graded module $\sections{\cotangent{\cM}}$ is often denoted by $\Omega^{1}(\cM)$, and its elements are called differential $1$-forms on $\cM$.

As in the tangent bundle case, assume that $(\cM,Q)$ is of positive amplitude and identify it with $(M,L,\lambda)$.
Let $\Omega^{1}(L)$ denote the space of differential $1$-forms on $L$, where $L$ is seen as a graded manifold.

The $\sections{SL^{\vee}}$-linear dual of the exact sequence~\eqref{eq:TangentSES} is an exact sequence of graded $\sections{SL^{\vee}}$-modules
\begin{equation}\label{eq:CotangentSES}
\begin{tikzcd}
0 \arrow{r}& \sections{SL^{\vee}\tensor \cotangent{M}} \arrow{r}{\pi^{\star}} & \Omega^{1}(L) \arrow{r}& \sections{SL^{\vee}\tensor L^{\vee}} \arrow{r}& 0\, ,
\end{tikzcd}
\end{equation}
where $\pi^{\star}$ denotes the $\sections{SL^{\vee}}$-linear dual of $\pi_{\ast}$ in~\eqref{eq:TangentSES}. Moreover, upon a choice of $\tangent{M}$-connection on $L$, we obtain
\[\sections{SL^{\vee}\tensor (L^{\vee}\oplus \cotangent{M})}\cong \Omega^{1}(L)\, .\] 
This yields the following cotangent analogue of Proposition~\ref{prop:Tangent}.
Since $L=\bigoplus_{k=1}^{b}L^{k}$, its dual is $L^{\vee}=\bigoplus_{k=1}^{b}(L^{\vee})^{-k}$, where $(L^{\vee})^{-k}\cong(L^{k})^{\vee}$ has degree $-k$.

\begin{proposition}\label{prop:Cotangent}
For the same choice of connection as in Proposition~\ref{prop:Tangent},
there exists an isomorphism of DG modules over $(\smooth{\cM},Q)=(\sections{SL^{\vee}},\lambda^{\transpose})$:
\[(\Omega^{1}(\cM),\liederivative{Q}) \cong (\sections{SL^{\vee}\tensor (L^{\vee}\oplus \cotangent{M})}, D_{\lambda}^{\transpose})\, ,\]
where $D_{\lambda}^{\transpose}$ is the differential induced by
$D_{\lambda}$ on the $\sections{SL^{\vee}}$-linear dual. Using the
indexing convention $(L^{\vee})^{0}:=\cotangent{M}$, we have $D_{\lambda}^{\transpose}
=({d_{ij}}^{\transpose})_{0\leq i,j\leq b}$ with
\[
{d_{ij}}^{\transpose}:
\sections{SL^{\vee}\tensor (L^{\vee})^{-i}}
\to
\sections{SL^{\vee}\tensor (L^{\vee})^{-j}} \, .
\]
In particular, ${d_{ii}}^{\transpose}=Q\tensor\id$.
\end{proposition}

\section{Tensor fields and weak equivalences}

In this section, we prove that the cohomology of the complex of $(p,q)$-tensor fields on DG manifolds of positive amplitude is
invariant under weak equivalences. We begin with the case of vector fields.

\subsection{Vector fields and weak equivalences}
The main result of this subsection is the following:

\begin{theorem} \label{thm:main1}
Let $(\cM,Q)$ and $(\cN,R)$ be DG manifolds of positive amplitude, and let $\Psi:(\cM,Q)\to (\cN,R)$ be a weak equivalence.
Then the cohomology groups of vector fields $H^{\bullet}(\XX(\cM),\liederivative{Q})$ and $H^{\bullet}(\XX(\cN),\liederivative{R})$ are isomorphic.
\end{theorem}

We prove Theorem~\ref{thm:main1} by constructing a pair of quasi-isomorphisms connecting the complexes of vector fields $(\XX(\cM),\liederivative Q)$ and $(\XX(\cN),\liederivative R)$.

Note that a morphism $f:M\to N$ between smooth manifolds naturally induces a morphism of vector bundles $f_{\ast}:\tangent{M}\to f^{\ast}\tangent{N}$ over $M$, where $f^{\ast}\tangent{N}$ is the pullback of $\tangent{N}$ along $f$. 
In terms of sections, $f$ naturally induces a morphism of $\smooth{M}$-modules $f_{\ast}:\XX(M)\to \sections{f^{\ast}\tangent{N}}$. 

It is well known (see for instance \cite[Section~11.53]{MR4221224}) that there is an isomorphism of $\smooth{M}$-modules
\[\sections{f^{\ast}\tangent{N}}\cong \smooth{M}\tensor_{\smooth{N}}\XX(N)\cong \Der_{f^{\ast}}(\smooth{N},\smooth{M})\, ,\]
where $\Der_{f^{\ast}}(\smooth{N},\smooth{M})$ consists of $\RR$-linear maps $Y:\smooth{N}\to \smooth{M}$ such that the Leibniz rule twisted by $f^{\ast}$ holds: for any $\alpha,\beta \in \smooth{N}$,
\[Y(\alpha \beta)=Y(\alpha)f^{\ast}(\beta)+f^{\ast}(\alpha)Y(\beta) \, .\]
Under this identification, the induced morphism of $\smooth{M}$-modules
\begin{equation}\label{eq:f-ast}
f_{\ast}:\XX(M)\to \sections{ f^{\ast}\tangent{N}}\cong \Der_{f^{\ast}}(\smooth{N},\smooth{M})
\end{equation}
satisfies
\[
f_{\ast}(X):\alpha \mapsto X(f^{\ast}(\alpha))
\]
for all $X\in \XX(M)$ and $\alpha\in \smooth{N}$.
Moreover, there is a natural morphism of $\smooth{N}$-modules
\begin{equation}\label{eq:I}
I:\XX(N) \to \sections{f^{\ast}\tangent{N}}
\end{equation}
which can be described in two equivalent ways: for all $Z\in \XX(N)$,
\begin{enumerate}
\item under the identification $\sections{ f^{\ast}\tangent{N}}\cong \smooth{M}\tensor_{\smooth{N}}\XX(N)$,
\[ I(Z)=1\tensor Z\, ,\]
\item under the identification $\sections{ f^{\ast}\tangent{N}}\cong \Der_{f^{\ast}}(\smooth{N},\smooth{M})$,
\[ I(Z): \alpha \mapsto f^{\ast}(Z(\alpha))\]
for all $\alpha \in \smooth{N}$.
\end{enumerate}
Note that $\sections{f^{\ast}\tangent{N}}$ is regarded as $\smooth{N}$-module via $f^{\ast}:\smooth{N}\to \smooth{M}$.
Geometrically, $I$ sends a section $s:N\to\tangent N$ to the pullback section $I(s):M\to f^{\ast}\tangent{N}$ defined by $I(s)(p)=(p,s(f(p)))$.

An analogous construction holds for DG manifolds. Given a morphism $\Psi:(\cM,Q)\to (\cN,R)$ of DG manifolds, there is an isomorphism of DG $(\smooth{\cM},Q)$-modules (see \cite[Lemma~3.5]{MR3320307} for the corresponding result in the setting of supermanifolds)
\begin{equation}\label{eq:DerTen}
 (\smooth{\cM}\tensor_{\smooth{\cN}}\XX(\cN), Q\tensor \id + \id \tensor \liederivative{R}) \cong (\Der_{\Psi^{\ast}}(\smooth{\cN},\smooth{\cM}), \liederivative{Q,R})
\end{equation}
where $\liederivative{Q,R}(Y):=Q\circ Y - (-1)^{\degree{Y}}Y\circ R$ for $Y\in \Der_{\Psi^{\ast}}(\smooth{\cN},\smooth{\cM})$. 
Both underlying graded $\smooth{\cM}$-modules in~\eqref{eq:DerTen} identify naturally with $\sections{\Psi^\ast\tangent{\cN}}$. We denote the resulting DG module by $(\sections{\Psi^\ast\tangent{\cN}},\liederivative{Q,R})$.
Using this notation, we obtain the following pair of cochain maps:
\begin{equation}\label{eq:PairMorph}
\begin{tikzcd}
(\XX(\cM),\liederivative{Q}) \arrow{r}{\Psi_{\ast}} &(\sections{\Psi^{\ast}\tangent{\cN}},\liederivative{Q,R})& (\XX(\cN), \liederivative{R}) \arrow{l}[swap]{I}
\end{tikzcd}
\end{equation}
where $\Psi_{\ast}$ is defined as in~\eqref{eq:f-ast}, with $\Psi$ in place of $f$, and $I$ is defined as in~\eqref{eq:I}. Note that $\Psi_{\ast}$ is a morphism of DG $(\smooth{\cM},Q)$-modules, whereas $I$ is a morphism of DG $(\smooth{\cN},R)$-modules.

To prove Theorem~\ref{thm:main1}, it suffices to restrict our attention to acyclic linear fibrations (that is, weak equivalences that are also linear fibrations) rather than arbitrary weak equivalences.

\begin{lemma}\label{lem:ALF}
If Theorem~\ref{thm:main1} holds for all acyclic linear fibrations $\Psi$, then it also holds for all weak equivalences $\Psi$.
\end{lemma}
\begin{proof}
By Lemma~\ref{lem:LinFib}, if Theorem~\ref{thm:main1} holds for all acyclic linear fibrations, then it holds for all acyclic fibrations.

By Theorem~\ref{thm:CatFib} and Corollary~\ref{cor:Fac}, if $\Psi:(\cM,Q)\to (\cN,R)$ is a weak equivalence between DG manifolds of positive amplitude, then $(\cM,Q)$ and $(\cN,R)$ are linked by a pair of acyclic fibrations. 
Since each acyclic fibration induces an isomorphism between cohomology groups of vector fields, 
the cohomology groups $H^{\bullet}(\XX(\cM),\liederivative{Q})$ and $H^{\bullet}(\XX(\cN),\liederivative{R})$ are isomorphic.
This completes the proof.
\end{proof}

\begin{remark}\label{rem:ALF} 
The same reasoning gives analogous reductions for tensor fields (Theorems~\ref{thm:main2} and~\ref{thm:main3}) and for Atiyah classes (Theorem~\ref{thm:main4}).
\end{remark}

Theorem~\ref{thm:main1} is a consequence of Lemma~\ref{lem:ALF} and the following theorem.
\begin{theorem}\label{thm:main1-1}
Let $(\cM,Q)$ and $(\cN,R)$ be DG manifolds of positive amplitude.
If $\Psi:(\cM,Q)\to (\cN,R)$ is an acyclic linear fibration, 
then both cochain maps in~\eqref{eq:PairMorph} are quasi-isomorphisms.
\end{theorem}

In order to prove Theorem~\ref{thm:main1-1}, it suffices to show that both $I$ and $\Psi_{\ast}$ are quasi-isomorphisms.
We first treat the map $I$, whose quasi-isomorphism property follows from Corollary~\ref{cor:KeyLem2}.

\begin{proposition}\label{prop:IMap}
Let $\Psi:(\cM,Q)\to (\cN,R)$ be a weak equivalence between DG manifolds of positive amplitude.
The morphism of DG $(\smooth{\cN},R)$-modules
\[
I:(\XX(\cN),\liederivative{R}) \to (\sections{\Psi^{\ast}\tangent{\cN}},\liederivative{Q,R})
\]
in~\eqref{eq:PairMorph} is a quasi-isomorphism.
\end{proposition}

\begin{proof}
By Proposition~\ref{prop:WEQI}, the induced map 
$\eta:=\Psi^{\ast}:(\smooth{\cN},R) \to (\smooth{\cM}, Q)$
is a quasi-isomorphism of DG algebras. Viewing $(\smooth{\cM},Q)$ as a DG $(\smooth{\cN},R)$-module, $\eta$ is a quasi-isomorphism of DG $(\smooth{\cN},R)$-modules. 

Applying Corollary~\ref{cor:KeyLem2} to the DG vector bundle $(\tangent{\cN}, \liederivative{R})$, we obtain a quasi-isomorphism
\[ (\XX(\cN),\liederivative{R}) \to (\smooth{\cM}\tensor_{\smooth{\cN}}\XX(\cN), Q\tensor \id + \id \tensor \liederivative{R}) \cong (\sections{\Psi^{\ast}\tangent{\cN}},\liederivative{Q,R})\, .\]
Under the identification~\eqref{eq:DerTen}, this map is precisely $I$.
This completes the proof.
\end{proof}

We now turn to $\Psi_{\ast}$, whose quasi-isomorphism property will be proved in the following subsection.
\begin{proposition}\label{prop:PsiMap}
Under the hypotheses of Theorem~\ref{thm:main1-1}, 
the morphism of DG $(\smooth{\cM},Q)$-modules 
\[
\Psi_{\ast}:(\XX(\cM),\liederivative{Q}) \to (\sections{\Psi^{\ast}\tangent{\cN}},\liederivative{Q,R})
\]
in~\eqref{eq:PairMorph} is a quasi-isomorphism.
\end{proposition}

\subsection{Proof of Proposition~\ref{prop:PsiMap}}

By Proposition~\ref{prop:CatBundle}, the statement in Proposition~\ref{prop:PsiMap} can be reformulated in terms of bundles of positively graded curved $L_{\infty}[1]$ algebras. 

First, we reformulate $\Psi_{\ast}$ in terms of bundles of positively graded curved $L_{\infty}[1]$ algebras. For the following lemma, we consider linear morphisms.
\begin{lemma}\label{lem:PhiStar}
Let $\Phi=(f,\phi_{1}):(M,L,\lambda) \to (N,E,\mu)$, with $f:M\to N$ and $\phi_{1}:L\to f^{\ast}E$, be a linear morphism in the category $\cLinfCat$, corresponding to a linear morphism $\Psi:(\cM,Q)\to (\cN,R)$ in the category $\DGMfdCat$. Given a $\tangent{M}$-connection on $L$ and a $\tangent{N}$-connection on $E$, the induced map $\Phi_{\ast}$, which corresponds to $\Psi_{\ast}$, is a DG $(\sections{SL^{\vee}}, \lambda^{\transpose})$-module map
\[
\Phi_{\ast}=
\begin{bmatrix}
f_{\ast}  & 0 \\
\gamma &  \phi_{1}
\end{bmatrix}
:  (\sections{SL^{\vee}\tensor (\tangent{M}\oplus L)}, D_{\lambda}) \to (\sections{SL^{\vee} \tensor (f^{\ast}\tangent{N}\oplus f^{\ast}E)}, D_{\mu}) \, ,
\]
where $D_\lambda$ is as in Proposition~\ref{prop:Tangent} and, by abuse of notation, $D_\mu$ denotes the differential induced by extension of scalars along $\Phi^\ast$ from the differential $D_\mu$ in Proposition~\ref{prop:Tangent}.
The map
\[\gamma: \sections{SL^{\vee}\tensor \tangent{M}} \to \sections{SL^{\vee}\tensor f^{\ast}E}\]
is an $\sections{SL^{\vee}}$-linear map that depends on the choice of connections.
\end{lemma}

\begin{proof}
Although the statement largely follows from the corresponding morphism $\Psi$ in $\DGMfdCat$ and Proposition~\ref{prop:Tangent}, the explicit formula for $\Phi_{\ast}$ requires verification.

Recall that $\sections{SE^{\vee}}=\sections{N;SE^{\vee}}$ is the space of smooth sections of $SE^{\vee}$ on $N$, and similarly that $\sections{SL^{\vee}}=\sections{M;SL^{\vee}}$ is the space of smooth sections of $SL^{\vee}$ on $M$.
The graded algebra homomorphism $\Phi^{\ast}:\sections{SE^{\vee}} \to \sections{SL^{\vee}}$ corresponding to 
$\Psi^{\ast}:\smooth{\cN}\to \smooth{\cM}$ is characterised by the maps 
\[ f^{\ast}:\smooth{N}\to \smooth{M}, \qquad \phi_{1}^{\ast}: \sections{f^{\ast}E^{\vee}}\to \sections{L^{\vee}}\]
induced by $f:M\to N$ and $\phi_{1}:L\to f^{\ast}E$. In particular, for $\xi\in \sections{N;E^{\vee}}$, we have $\Phi^{\ast}(\xi)=\phi_{1}^{\ast}(f^{\ast}\xi)$. In the remainder of the proof, by abuse of notation, we often write 
$\phi_1^\ast(\xi)=\phi_1^\ast(f^\ast\xi)=\Phi^\ast(\xi)$.

By the isomorphisms~\eqref{eq:SLTL} and~\eqref{eq:DerTen}, a choice of $\tangent{N}$-connection $\nabla^{E}$ on $E$ induces graded $\sections{SL^{\vee}}$-linear isomorphisms
\begin{equation}\label{eq:LE}
\sections{SL^{\vee}\tensor (f^{\ast}\tangent{N} \oplus f^{\ast}E)} \cong  \sections{SL^{\vee}}\tensor_{\sections{SE^{\vee}}} \sections{SE^{\vee}\tensor(\tangent{N}\oplus E)} \cong \Der_{\Phi^{\ast}}(\sections{SE^{\vee}}, \sections{SL^{\vee}})
\end{equation}
whose composite is characterised by the formulae in~\eqref{eq:VFBase}.
More concretely, for $\widetilde{Y}\in \sections{f^{\ast}\tangent{N}}$ and $\widetilde{e}\in \sections{f^{\ast}E}$ arising from $Y\in \XX(N)$ and $e\in \sections{E}$, respectively, the isomorphism~\eqref{eq:LE} is characterised by
\[ \widetilde{Y} \mapsto \begin{cases}
\beta \mapsto \Phi^{\ast}(Y(\beta)) \\
\xi \mapsto \Phi^{\ast}(\widetilde{\nabla}^{E}_{Y}\xi)
\end{cases}
, \qquad \widetilde{e} \mapsto 
\begin{cases}
\beta \mapsto 0 \\
 \xi \mapsto \Phi^{\ast}(\iota_{e}\xi)
 \end{cases},
  \qquad 
  \forall \beta \in \smooth{N}, \forall \xi \in \sections{E^{\vee}}\, ,\]
where $\widetilde{\nabla}^{E}$ is the induced $\tangent{N}$-connection on $SE^{\vee}$. 
These assignments extend to a well-defined graded $\sections{SL^{\vee}}$-linear map.

Given a $\tangent{M}$-connection $\nabla^{L}$ on $L$ and a $\tangent{N}$-connection $\nabla^{E}$ on $E$,
we have the following commutative diagram
\[
\begin{tikzcd}
\sections{SL^{\vee}\tensor (\tangent{M}\oplus L)} \arrow{rr}{\Phi_{\ast}} \arrow[no head]{d}{\rotatebox{90}{$\sim$}}&& \sections{SL^{\vee} \tensor (f^{\ast}\tangent{N}\oplus f^{\ast}E)} \arrow[no head]{d}{\rotatebox{90}{$\sim$}}\\
\Der(\sections{SL^{\vee}}) \arrow{rr}{\widetilde{\Phi}_{\ast}} && \Der_{\Phi^{\ast}}(\sections{SE^{\vee}}, \sections{SL^{\vee}}) 
\end{tikzcd}
\]
where $\widetilde{\Phi}_{\ast}$ is defined by~\eqref{eq:f-ast}, with $\Phi$ in place of $f$.

Consider $l\in \sections{L}$. Then we may write $\phi_{1}(l)\in \sections{f^{\ast}E}\cong \smooth{M}\tensor_{\smooth{N}}\sections{E}$ as
\[ \phi_{1}(l)=\sum_{i} \alpha_{i}\tensor e_{i}\]
for some $\alpha_{i}\in \smooth{M}$ and $e_{i}\in \sections{E}$. Then, it is straightforward to check that
\[\big( \widetilde{\Phi}_{\ast}(\iota_{l}) \big)(\beta)=\iota_{l}(f^{\ast}(\beta))=0, \quad \big( \widetilde{\Phi}_{\ast}(\iota_{l}) \big)(\xi)=\iota_{l}( \phi_{1}^{\ast}(\xi))=\sum_{i} \alpha_{i}\cdot \Phi^{\ast}(\iota_{e_{i}}\xi)\]
for $\beta \in \smooth{N}$ and $\xi \in \sections{E^{\vee}}$. Thus, we have
\[\Phi_{\ast}(l) = \phi_{1}(l)\in \sections{f^{\ast}E}\subset \sections{SL^{\vee} \tensor (f^{\ast}\tangent{N}\oplus f^{\ast}E)}\, .\]

Next, consider $X\in \XX(M)$. Similarly, we may write 
\[f_{\ast}(X)=\sum_{i}\alpha_{i}\tensor Y_{i}\]
for some $\alpha_{i}\in \smooth{M}$ and $Y_{i}\in \XX(N)$. Then, one can verify that
\begin{gather*}
 \big( \widetilde{\Phi}_{\ast}(\widetilde{\nabla}^{L}_{X}) \big)(\beta)=\widetilde{\nabla}^{L}_{X}(f^{\ast}(\beta))=(f_{\ast}X)(\beta),\\\big( \widetilde{\Phi}_{\ast}(\widetilde{\nabla}^{L}_{X}) \big)(\xi)=\widetilde{\nabla}^{L}_{X}( \phi_{1}^{\ast}(\xi))=\sum_{i}\alpha_{i}\cdot \phi_{1}^{\ast}(\widetilde{\nabla}^{E}_{Y_{i}}\xi) + \big( \widetilde{\nabla}^{L}_{X}( \phi_{1}^{\ast}(\xi)) - \sum_{i}\alpha_{i}\cdot \phi_{1}^{\ast}(\widetilde{\nabla}^{E}_{Y_{i}}\xi)\big)
\end{gather*}
for $\beta\in\smooth{N}$ and $\xi \in \sections{E^{\vee}}$. Thus, we have
\[ \Phi_{\ast}(X)=f_{\ast}X + \gamma(X) \in 
\sections{SL^{\vee} \tensor (f^{\ast}\tangent{N}\oplus f^{\ast}E)}\, ,\]
where 
\begin{equation}\label{eq:Gamma}
\gamma(X):\xi \mapsto \widetilde{\nabla}^{L}_{X}( \phi_{1}^{\ast}(\xi)) - \sum_{i}\alpha_{i}\cdot \phi_{1}^{\ast}(\widetilde{\nabla}^{E}_{Y_{i}}\xi) \, .
\end{equation}
Equivalently,
\[\gamma(X)=\nabla^{f^{\ast}E}_{X}\circ \phi_{1} -\phi_{1}\circ\nabla_{X}^{L}\in \sections{L^{\vee}\tensor f^{\ast}E}\subset \sections{SL^{\vee}\tensor f^{\ast}E}\, .\]
This completes the proof.
\end{proof}

Assume that the linear morphism $\Phi$ is a fibration. Then the bundle map $\phi_{1}:L\to f^{\ast}E$ is surjective. This yields an exact sequence of vector bundles over $M$:
\[\begin{tikzcd}
 0 \arrow{r}& K \arrow{r}&  L \arrow{r}{\phi_{1}}& f^{\ast}E \arrow{r}& 0 \, .
 \end{tikzcd}
 \]
Upon a choice of a splitting, we may identify 
$L\cong K\oplus f^{\ast}E$.
Under this identification, 
a $\tangent{M}$-connection $\nabla^{K}$ on $K$ and a $\tangent{N}$-connection $\nabla^{E}$ on $E$ determine a $\tangent{M}$-connection
\begin{equation}\label{eq:NablaL}
\nabla^{L}=\nabla^{K} \oplus \nabla^{f^{\ast}E}
\end{equation}
where $\nabla^{f^{\ast}E}$ denotes the pullback connection of $\nabla^{E}$ along $f:M\to N$.
With these choices of connections, 
by~\eqref{eq:Gamma}, we have $\gamma(X)=0$ for all $X\in\XX(M)$. Therefore, we have proved the following:
\begin{corollary}\label{cor:Ker1}
Under the hypothesis of Lemma~\ref{lem:PhiStar}, assume furthermore that $\Phi$ is a fibration. Then, for any choice of a $\tangent{M}$-connection on $K:=\ker \phi_{1}$ and a $\tangent{N}$-connection on $E$, there exists a $\tangent{M}$-connection on $L$ such that $\gamma=0$.
\end{corollary}

Now we reformulate Proposition~\ref{prop:PsiMap} in terms of bundles of positively graded curved $L_{\infty}[1]$ algebras.
\begin{lemma}\label{lem:Equiv}
Let $\Phi=(f,\phi_{1}):(M,L,\lambda) \to (N,E,\mu)$, with $f:M\to N$ and $\phi_{1}:L\to f^{\ast}E$, be an acyclic linear fibration in the category $\cLinfCat$ corresponding to an acyclic linear fibration $\Psi:(\cM,Q)\to (\cN,R)$ in the category $\DGMfdCat$. 
Choose connections as in Corollary~\ref{cor:Ker1}, so that
$\gamma=0$. Then the following are equivalent.
\begin{enumerate}
\item Proposition~\ref{prop:PsiMap} holds, i.e., $\Psi_{\ast}$ is a quasi-isomorphism. \label{item:Psi}
\item The DG module map $\Phi_{\ast}$ in Lemma~\ref{lem:PhiStar} with $\gamma=0$  \label{item:Phi}
\[\Phi_{\ast}:=
\begin{bmatrix}
f_{\ast}  & 0 \\
0 &  \phi_{1}
\end{bmatrix}
:  (\sections{SL^{\vee}\tensor (\tangent{M}\oplus L)}, D_{\lambda}) \to (\sections{SL^{\vee} \tensor (f^{\ast}\tangent{N}\oplus f^{\ast}E)}, D_{\mu}) \, ,
\]
is a quasi-isomorphism.
\item The DG module $\ker \Phi_{\ast}$ is acyclic. \label{item:KerPhi}
\end{enumerate}
\end{lemma}
\begin{proof}
The equivalence between~\eqref{item:Psi} and~\eqref{item:Phi} follows from Lemma~\ref{lem:PhiStar} and Corollary~\ref{cor:Ker1}. 

We prove the equivalence between~\eqref{item:Phi} and~\eqref{item:KerPhi}. 
Since $\Phi=(f,\phi_{1})$ is a linear fibration, both bundle maps $f_{\ast}:\tangent{M}\to f^{\ast}\tangent{N}$ and $\phi_{1}:L\to f^{\ast}E$ are surjective. Thus, $\Phi_{\ast}$ is a surjective morphism of cochain complexes, and induces an exact sequence of cochain complexes
\[ 0 \to \ker \Phi_{\ast} \to (\sections{SL^{\vee}\tensor (\tangent{M}\oplus L)}, D_{\lambda}) \xto{\Phi_{\ast}} (\sections{SL^{\vee} \tensor (f^{\ast}\tangent{N}\oplus f^{\ast}E)}, D_{\mu}) \to 0\, .\]
By a standard argument in homological algebra (see for instance \cite[Theorem~1.3.1]{MR1269324}), the cochain complex $\ker \Phi_{\ast}$ is acyclic if and only if $\Phi_{\ast}$ is a quasi-isomorphism. This proves the equivalence between~\eqref{item:Phi} and~\eqref{item:KerPhi}.
\end{proof}

\subsubsection{Kernel of $\Phi_{\ast}$}

Lemma~\ref{lem:Equiv} reduces the proof of Proposition~\ref{prop:PsiMap} to showing that $\ker \Phi_{\ast}$ is acyclic. 
To do so, we analyse the DG module structure of $\ker \Phi_{\ast}$. In the following lemma, we only require $\Phi$ to be a linear fibration.

\begin{lemma}\label{lem:KerPhi}
Let $\Phi=(f,\phi_{1}):(M,L,\lambda)\to (N,E,\mu)$ be a linear fibration in the category $\cLinfCat$, where $f:M\to N$ and $\phi_{1}:L\to f^{\ast}E$.
Suppose that the $\tangent{M}$-connection $\nabla=\nabla^{L}$ on $L$ is defined by Eq.~\eqref{eq:NablaL}. Then,
\[\ker \Phi_{\ast} = (\sections{SL^{\vee}\tensor (\ker f_{\ast}\oplus \ker \phi_{1})}, \overline{D})\]
is a DG $(\sections{SL^{\vee}},\lambda^{\transpose})$-module,
where the differential $\overline{D}=D_{\lambda} \big|_{\sections{SL^{\vee}\tensor (\ker f_{\ast}\oplus \ker \phi_{1})}}$ is defined by the restriction of $D_{\lambda}$ in Proposition~\ref{prop:Tangent}.
\end{lemma}
\begin{proof}
The lemma is an immediate consequence of 
Lemma~\ref{lem:PhiStar}, Corollary~\ref{cor:Ker1}, and Proposition~\ref{prop:Tangent}.
\end{proof}

We compute explicitly those components of the differential $\overline{D}$ that are required to prove Proposition~\ref{prop:PsiMap}.
\begin{corollary}\label{cor:KerPhi}
In the setting of Lemma~\ref{lem:KerPhi}, let $K^{0}:=\ker f_{\ast}$ and $\ker \phi_{1}=\bigoplus_{i=1}^{b}K^{i}$. 
Then the differential $\overline{D}=(\overline{d}_{ij})_{0\leq i,j \leq b}$ with 
\[\overline{d}_{ij}:\sections{SL^{\vee}\tensor K^{j}}\to \sections{SL^{\vee}\tensor K^{i}}\]
satisfies
\begin{enumerate}
\item $\overline{d}_{ij}=0$ if $i<j$;
\item $\overline{d}_{ii}=\lambda^{\transpose}\tensor \id$ for all $i$;
\item the map $\overline{d}_{10}$ is the graded $\sections{SL^{\vee}}$-linear map defined by
\[\overline{d}_{10}(X)= -\nabla^{L}_{X}\lambda_{0}, \quad \forall X\in \sections{K^{0}} ;\]
\item for each $j=1,\ldots,b-1$, the map $\overline{d}_{j+1,j}$ is the graded $\sections{SL^{\vee}}$-linear map defined by
\[\overline{d}_{j+1,j}(k)= -\lambda_{1}(k), \quad \forall k\in \sections{K^{j}}\, .\]
\end{enumerate}
\end{corollary}

\begin{proof}
The identities $\overline{d}_{ij}=0$ if $i<j$ and $\overline{d}_{ii}=\lambda^{\transpose}\tensor \id$ follow directly from  Proposition~\ref{prop:Tangent}, using the fact that $Q=\lambda^{\transpose}$ in our setting.

If $X\in \sections{K^{0}}\subset \XX(M)$, then, by Proposition~\ref{prop:Tangent}~\eqref{item:3}, we have
\[\overline{d}_{10}(X)=\pr_{1}\circ [\lambda,\nabla_{X}]=-\nabla_{X}\lambda_{0}\, ,\]
where $\pr_{1}:\sections{SL}\to \sections{L^{1}}$ denotes the natural projection, and $\lambda$ is viewed as a $\smooth{M}$-coderivation $\sections{SL}\to \sections{SL}$.
Since $f_{\ast}(X)=0$, Eq.~\eqref{eq:NablaL} implies
\[\phi_{1}(\overline{d}_{10}(X))=-\phi_{1}(\nabla_{X}\lambda_{0})=-\nabla^{f^{\ast}E}_{X}\big(\phi_{1}(\lambda_{0})\big)=-\nabla^{f^{\ast}E}_{X}(f^{\ast}\mu_{0})=0\]
which shows that $\overline{d}_{10}(X)\in \sections{K^{1}}$.

For $k\in \sections{K^{j}}\subset \sections{L^{j}}$ with $j>0$, by Proposition~\ref{prop:Tangent}~\eqref{item:4}, we have
\[\overline{d}_{j+1,j}(k)=(-1)^{j+1}\pr_{j+1}\circ(\iota_k\lambda)=-\lambda_1(k)\, . \]
Since $\Phi$ is a linear morphism, the relation $\phi_{1} (\lambda_{1}(k))=\mu_{1} (\phi_{1}(k))=0$ shows that $\overline{d}_{ij}(k)\in \sections{K^{i}}$ for $i=j+1$. This completes the proof.
\end{proof}

Since $\Phi$ is a linear fibration, for each open subset $U\subset M$, we have
\begin{equation}\label{eq:KerPhiU}
\ker\Phi_{\ast}|_{U}= \big( \sections{U;SL^{\vee}\tensor(\ker f_{\ast}\oplus\ker\phi_{1})}, \overline D|_{U} \big).
\end{equation}
The assignment $U\mapsto\ker\Phi_{\ast}|_U$ forms a locally free sheaf of DG modules over the sheaf of DG algebras $(\cA,\lambda^{\transpose})$, 
where $\cA(U)=\sections{U;SL^{\vee}}$. 
Its complex of global sections is the DG module $\ker\Phi_{\ast}$ of Lemma~\ref{lem:KerPhi}.

\begin{remark}\label{rem:KerPsi}
This sheaf structure on $\ker \Phi_{\ast}$ can also be described geometrically, in terms of DG manifolds. Let $\Psi:(\cM,Q)\to (\cN,R)$ be a linear fibration corresponding to $\Phi$. 
Then the induced map
\[\Psi_{\ast}:(\tangent{\cM},\liederivative{Q}) \to (\Psi^{\ast}\tangent{\cN}, \liederivative{Q,R})\]
is a surjective morphism of DG vector bundles over $(\cM,Q)$. Consequently, $\ker \Psi_{\ast}$ is itself a DG vector bundle over $(\cM,Q)$. 
By Remark~\ref{rem:Sheaf}, this kernel $\ker\Psi_{\ast}$, which corresponds to $\ker \Phi_{\ast}$, admits the structure of a locally free sheaf of DG $(\cA,Q)$-modules over $M$. The resulting sheaf structure on $\ker \Phi_{\ast}$ agrees with the one described above, which was induced by the DG module structure over $(\sections{SL^{\vee}},\lambda^{\transpose})$.
\end{remark}

Let $\cO_{M}$ denote the sheaf of algebras of smooth functions on $M$. 
Since $(\cA,\lambda^{\transpose})$ is a sheaf of DG $\cO_M$-algebras, the sheaf $\ker\Phi_{\ast}$ is a sheaf of DG $\cO_M$-modules.
The following lemma (see \cite[Lemma~3.9]{MR5026679}) implies that the acyclicity of $\ker \Phi_{\ast}$ can be verified locally.

\begin{lemma}[\cite{MR5026679}]\label{lem:Locality}
Let $\cO_{M}$ be a sheaf of algebras of smooth functions on a smooth manifold $M$ and let $(\cE,\cQ_{\cE})$ be a sheaf of DG $\cO_{M}$-modules. Then the cohomology presheaf
\[ U \mapsto H^{\bullet}(\cE(U),\cQ_{\cE}(U)):= \frac{\ker (\cQ_{\cE}(U))}{\image (\cQ_{\cE}(U))}\]
is, in fact, a sheaf of graded $\cO_{M}$-modules. 
\end{lemma}

Applying Lemma~\ref{lem:Locality} to the sheaf $\ker \Phi_{\ast}$, we obtain the following consequence.
\begin{corollary}\label{cor:Locality}
The DG $(\sections{SL^{\vee}},\lambda^{\transpose})$-module $\ker \Phi_{\ast}$ described in Lemma~\ref{lem:KerPhi} is acyclic if, for every point $p\in M$, there exists an open neighbourhood $U$ of $p$ such that the restriction $\ker \Phi_{\ast}|_{U}$ is acyclic.
\end{corollary}
\begin{proof}
As described in Remark~\ref{rem:KerPsi} and above, we may view $\ker \Phi_{\ast}$ as a sheaf of DG $\cO_{M}$-modules over $M$. 
By Lemma~\ref{lem:Locality}, the cohomology $H(\ker \Phi_{\ast})$ is a sheaf. 
Note that a cochain complex is acyclic if and only if its cohomology vanishes. 

Assume that for every point $p\in M$, there is an open neighbourhood $U_{p}$ of $p$ such that $H(\ker \Phi_{\ast}(U_{p}))=0$.
Then for any global section $X\in H(\ker \Phi_{\ast}(M))$, 
its restriction to each open set $X|_{U_{p}}\in H(\ker \Phi_{\ast})(U_{p})=0$ must vanish. 
Since $H(\ker \Phi_{\ast})$ is a sheaf, the locality axiom implies $X=0$. 
Hence, $H(\ker \Phi_{\ast}(M))=0$, so the DG module $\ker \Phi_{\ast}$, considered as the complex of global sections, is acyclic. 
This completes the proof.
\end{proof}

By Lemma~\ref{lem:Equiv} and Corollary~\ref{cor:Locality}, it remains to show that for each point $p\in M$, there is an open neighbourhood $U$ of $p$ such that $\ker \Phi_{\ast}|_{U}$ is acyclic. For this, we distinguish two cases, depending on whether the point $p\in M$ is classical or not---see Definition~\ref{defn:WE}.

The case of a non-classical point is standard. Observe that the set of classical points $Z(\lambda_{0})$ is a closed subset of $M$, and hence $M\setminus Z(\lambda_{0})$ is open.
\begin{lemma}\label{lem:NonClassical}
Let $\Phi:(M,L,\lambda)\to (N,E,\mu)$ be a linear fibration in $\cLinfCat$. Then, for $U=M\setminus Z(\lambda_{0})$, the DG module $\ker \Phi_{\ast}|_{U}$ as in Eq.~\eqref{eq:KerPhiU} is acyclic.
\end{lemma}
\begin{proof}
Let $(\cM|_{U},Q|_{U})$ be the DG manifold of positive amplitude corresponding to $(U,L|_{U}, \lambda|_{U})$. Then $\ker \Phi_{\ast}|_{U}$ corresponds to a DG vector bundle over $(\cM|_{U},Q|_{U})$. 

It is well known (see for instance~\cite[Lemma~5.4]{arXiv:2303.11140}) that if a DG manifold of positive amplitude does not contain any classical points, then its DG algebra of functions is acyclic. Since $U$ consists of non-classical points, $(\smooth{\cM|_{U}},Q|_{U})$ is acyclic. By Corollary~\ref{cor:KeyLem3}, we may conclude that $\ker \Phi_{\ast}|_{U}$ is acyclic. This completes the proof.
\end{proof}

Now, we consider classical points. To prove that there exists an open neighbourhood $U$ of a classical point such that $\ker\Phi_{\ast}|_{U}$ is acyclic, we require $\Phi$ to be an acyclic linear fibration. 

Let $\Phi=(f,\phi_{1}):(M,L,\lambda)\to (N,E,\mu)$ be an acyclic linear fibration in the category $\cLinfCat$, with $f:M\to N$, $\phi_{1}:L\to f^{\ast}E$. 
Since $\Phi$ is a fibration, we may write $L=\bigoplus_{i=1}^{b}L^{i}$ and $E=\bigoplus_{i=1}^{b} E^{i}$ (that is, $E^{i}=0$ for $i>b$). 
Moreover, since $\Phi$ is a weak equivalence, the pair $(f,\phi_{1})$ induces a quasi-isomorphism between the tangent complexes at each classical point. 
More precisely, if $p\in M$ is a classical point, then the maps $f_{\ast}|_{p}$ and $\phi_{1}|_{p}$ appearing in the diagram
\[
\begin{tikzcd}
0 \arrow{r} & \tangentp{f(p)}{N} \arrow{r}{D_{f(p)}\mu_{0}} & E^{1}|_{f(p)} \arrow{r}{\mu_{1}|_{f(p)}} & E^{2}|_{f(p)} \arrow{r}{\mu_{1}|_{f(p)}} & \cdots \arrow{r}{\mu_{1}|_{f(p)}} & E^{b}|_{f(p)}\arrow{r} & 0\\
0 \arrow{r} & \tangentp{p}{M}\arrow{u}{f_{\ast}|_{p}} \arrow{r}{D_{p}\lambda_{0}} & L^{1}|_{p} \arrow{u}{\phi_{1}|_{p}} \arrow{r}{\lambda_{1}|_{p}} & L^{2}|_{p} \arrow{u}{\phi_{1}|_{p}}\arrow{r}{\lambda_{1}|_{p}} & \cdots\arrow{r}{\lambda_{1}|_{p}} & L^{b}|_{p}\arrow{r} \arrow{u}{\phi_{1}|_{p}}& 0
\end{tikzcd}
\]
induce a quasi-isomorphism of cochain complexes. Here, $L^{i}|_{p}$ denotes the fibre of $L^{i}$ at $p$, 
and similarly, $E^{i}|_{f(p)}$ denotes the fibre of $E^{i}$ at $f(p)$.
Note that, since $\Phi$ is a fibration, both maps $f_{\ast}|_{p}$ and $\phi_{1}|_{p}$ are surjective. Thus, by a standard argument in homological algebra (cf. proof of Lemma~\ref{lem:Equiv}), the kernel of $(f_{\ast}|_{p}, \phi_{1}|_{p})$ is acyclic. In other words, by writing $K^{0}:=\ker f_{\ast}$ and $K:=\ker \phi_{1}=\bigoplus_{i=1}^{b}K^{i}$, we have an exact sequence
\begin{equation}\label{eq:SESp}
\begin{tikzcd}
0 \arrow{r} & K^{0}|_{p} \arrow{r}{\overline{D_{p}\lambda_{0}}} & K^{1}|_{p} \arrow{r}{\overline{\lambda}_{1}|_{p}} & K^{2}|_{p} \arrow{r}{\overline{\lambda}_{1}|_{p}} & \cdots \arrow{r}{\overline{\lambda}_{1}|_{p}} & K^{b}|_{p} \arrow{r} & 0
\end{tikzcd}
\end{equation}
where $\overline{D_{p}\lambda_{0}} := (D_{p}\lambda_{0})|_{K^{0}|_{p}}$ and $\overline{\lambda}_{1}|_{p}:= \lambda_{1}|_{K|_{p}}$.

By Eq.~\eqref{eq:cL-infty}, the operator $\overline{\lambda}_{1}:K^{i}\to K^{i+1}$ may not extend the exact sequence~\eqref{eq:SESp} to an open neighbourhood of $p$; in general, $\overline{\lambda}_{1}^{2}$ does not necessarily vanish. 
However, by choosing an appropriate correction to $\overline{\lambda}_{1}$, one can extend~\eqref{eq:SESp} to an open neighbourhood $U$ of $p$. 
The key observation is that a morphism of vector bundles that is surjective (resp. isomorphic) at the fibre over $p$ remains surjective (resp. isomorphic) at every fibre in a neighbourhood of $p$.

\begin{lemma} \label{lem:U-Exact}
Let $\Phi=(f,\phi_{1}):(M,L,\lambda)\to (N,E,\mu)$ be an acyclic linear fibration in the category $\cLinfCat$, where $f:M\to N$ and $\phi_{1}:L\to f^{\ast}E$. Let $\nabla=\nabla^{L}$ be the $\tangent{M}$-connection on $L$ defined as in Eq.~\eqref{eq:NablaL}. For each classical point $p\in M$, 
there exists an open neighbourhood $U$ of $p$ and 
appropriate choices of $\delta_{i}:K^{i}|_{U}\to K^{i+1}|_{U}$ for $i=0,1,\ldots,b-1$ such that
\begin{equation}\label{eq:DeltaSES}
\begin{tikzcd}
0\arrow{r}& K^{0}|_{U} \arrow{r}{\delta_{0}}&  K^{1}|_{U} \arrow{r}{\delta_{1}} & K^{2}|_{U} \arrow{r}& \cdots \arrow{r}& K^{b-1}|_{U} \arrow{r}{\delta_{b-1}} & K^{b}|_{U} \arrow{r}& 0
\end{tikzcd}
\end{equation}
is an exact sequence of vector bundles over $U$ extending~\eqref{eq:SESp}.
In fact, the maps $\delta_{i}$ are defined by
\begin{equation}\label{eq:Deltas}
\delta_{0}= \overline{\nabla \lambda_{0}} - \eta_{1}\circ \overline{\lambda}_{1}\circ \overline{\nabla \lambda_{0}}, \qquad
\delta_{i}=\overline{\lambda}_{1}-\eta_{i+1}\circ (\overline{\lambda}_{1})^{2} ,\quad 1\leq i \leq b-1 \, .
\end{equation}
where $\overline{\nabla\lambda_{0}}(X):=\nabla^{L}_{X}\lambda_{0}$ for $X\in K^{0}$. Here, we set $\delta_b=\eta_{b}=0$ and the maps $\eta_{i}:K^{i+1}|_{U}\to K^{i}|_{U}$, $i=0,1,\ldots,b-1$, are bundle maps satisfying 
\begin{equation}\label{eq:Etas}
\eta_{i}\circ \eta_{i+1}=0, \qquad \eta_{0}\circ \delta_{0}=\id_{K^{0}|_{U}},\qquad  \delta_{i}\circ \eta_{i} + \eta_{i+1}\circ \delta_{i+1}=\id_{K^{i+1}|_{U}}.
\end{equation}
\end{lemma}

\begin{proof}
Suppose that there exists an open neighbourhood $U$ of $p$ such that $\delta_{i}$ and $\eta_{i}$ satisfy Eqs.~\eqref{eq:Deltas} and~\eqref{eq:Etas}. 
Then it is straightforward to check that $\delta_{i+1}\delta_{i}=0$ and~\eqref{eq:DeltaSES} is an exact sequence. Moreover, since the map $\nabla \lambda_{0}$ restricted to the fibre at $p$ satisfies
\[(\nabla \lambda_{0})|_{p}: X_{p}\mapsto D_{p}\lambda_{0}(X_{p})\]
for $X_{p}\in \tangentp{p}{M}$, by Eqs.~\eqref{eq:Deltas}, the maps $\delta_{i}$ restricted to the fibre at $p$ satisfy
\[\delta_{0}|_{p}=\overline{D_{p}\lambda_{0}}:K^{0}|_{p}\to K^{1}|_{p}, \qquad \delta_{i}|_{p}=\overline{\lambda}_{1}|_{p}:K^{i}|_{p}\to K^{i+1}|_{p}\, .\]
Therefore exact sequence~\eqref{eq:DeltaSES} extends exact sequence~\eqref{eq:SESp}.

We now construct an open neighbourhood $U$ of $p$ and bundle maps $\delta_{i}$, $\eta_{i}$
satisfying Eqs.~\eqref{eq:Deltas} and Eqs.~\eqref{eq:Etas}. 

Consider the bundle map $\overline{\lambda}_{1}:K^{b-1}\to K^{b}$ (and if $b=1$, consider $\overline{\nabla\lambda_{0}}$). 
By the exactness of~\eqref{eq:SESp}, there exists an open neighbourhood $U$ of $p$ such that the induced bundle map 
\[\delta_{b-1}:=\lambda_{1}|_{U}:K^{b-1}|_{U}\to K^{b}|_{U}\]
is surjective (an isomorphism if $b=1$, where $\delta_0:=\overline{\nabla\lambda_0}|_U$). 
This yields an exact sequence of vector bundles:
\[\begin{tikzcd}
 0 \arrow{r}& \ker \delta_{b-1} \arrow{r}& K^{b-1}|_{U} \arrow{r}{\delta_{b-1}}& K^{b}|_{U} \arrow{r}& 0 \, .
 \end{tikzcd}\]
We choose a splitting $\eta_{b-1}:K^{b}|_{U} \to K^{b-1}|_{U}$ of this exact sequence. 
These maps $\delta_{b-1}$ and $\eta_{b-1}$ satisfy Eqs.~\eqref{eq:Etas} for $i=b-1$ (where $\delta_{b}=0$).

We construct the maps $\delta_{i}$ and $\eta_{i}$ inductively starting from $i=b-1$ and proceeding downward to $i=0$.
Suppose we have already constructed bundle maps $\delta_{i}$ and $\eta_{i}$ for all $i>r$ over an open neighbourhood $\widetilde{U}$, 
satisfying Eqs.~\eqref{eq:Etas} for $i>r$. Our goal now is to construct $\delta_{r}$ and $\eta_{r}$ satisfying the same conditions. 
For notational simplicity, we henceforth assume that $\widetilde{U}=M$ and omit restrictions to $\widetilde{U}$.

Note that, by Eqs.~\eqref{eq:Deltas} and Eqs.~\eqref{eq:Etas} for $i=r+1$, we have
\[\delta_{r+1}=\delta_{r+1}\circ \eta_{r+1}\circ \delta_{r+1}=\delta_{r+1}\circ \eta_{r+1}\circ (\overline{\lambda}_{1}- \eta_{r+2}\circ (\overline{\lambda}_{1})^{2}) = \delta_{r+1}\circ \eta_{r+1}\circ \overline{\lambda}_{1}\, .\]
Together with $\delta_{r}$ defined as in Eqs.~\eqref{eq:Deltas}, we have
\[\delta_{r+1}\circ \delta_{r}=\delta_{r+1}\circ (\overline{\lambda}_{1}- \eta_{r+1}\circ (\overline{\lambda}_{1})^{2})=\delta_{r+1}\circ \overline{\lambda}_{1} - \delta_{r+1} \circ \overline{\lambda}_{1}=0\, ,\]
which implies that $\image \delta_{r}\subset \ker \delta_{r+1}$ (and similarly for $r=0$).
By the exactness of~\eqref{eq:SESp}, the map $\delta_{r}$, seen as a bundle map $\delta_{r}:K^{r}\to \ker \delta_{r+1}$, is surjective (isomorphism if $r=0$) in a neighbourhood $U$ of $p$. Moreover, over $U$, we obtain a short exact sequence of vector bundles
\[\begin{tikzcd}
 0\arrow{r}& \ker \delta_{r}|_{U} \arrow{r}& K^{r}|_{U} \arrow{r}{\delta_{r}|_{U}}&  \ker \delta_{r+1}|_{U} \arrow{r}& 0 \, .
 \end{tikzcd}\]
We choose a splitting $\widetilde{\eta}_{r}: \ker \delta_{r+1}|_{U}\to K^{r}|_{U}$ of this sequence of vector bundles. 
Composing with the projection
\[P_{r+1}:=\id_{K^{r+1}} - \eta_{r+1}\circ \delta_{r+1} : K^{r+1} \to \ker \delta_{r+1}\]
yields a bundle map over $U$
\[ \eta_{r}:=\widetilde{\eta}_{r}\circ P_{r+1}=\widetilde{\eta}_{r}-\widetilde{\eta}_{r}\circ \eta_{r+1}\circ \delta_{r+1} : K^{r+1}|_{U}\to K^{r}|_{U}\, ,\]
which, by a direct computation, satisfies Eqs.~\eqref{eq:Etas}. 
This completes the construction, proving the existence of $U$ and maps $\delta_{i}$, $\eta_{i}$ satisfying the required properties.
\end{proof}

Fix a classical point $p\in M$, and let $U$ be an open neighbourhood constructed in Lemma~\ref{lem:U-Exact}.
We now prove that the DG module $\ker \Phi_{\ast}|_{U}$ is acyclic. 
To simplify notation, we omit explicit restrictions to $U$: for instance, we write $K^{i}$ instead of $K^{i}|_{U}$, and $\sections{SL^{\vee}\tensor K^{j}}$ instead of $\sections{U; SL^{\vee}\tensor K^{j}}$.

Let $\kappa^{i}:=\ker \delta_{i}$ for $i=1,\ldots,b$ (with the convention that $\delta_{b}=0$).
For each $n=1,\ldots,b$, define the graded $\sections{SL^{\vee}}$-module
 \[\cE(n)=\sections{SL^{\vee}\tensor ( \bigoplus_{i=0}^{n-1}K^{i}\oplus \kappa^{n} )}\, ,\]
equipped with a degree $+1$ operator
\[
\overline{D(n)}=(\overline{d(n)}_{ij})_{0\leq i,j \leq n} :\cE(n) \to \cE(n)\, ,
\]
whose component
\[
\overline{d(n)}_{ij}: \sections{SL^{\vee}\tensor K^{j}} \to \sections{SL^{\vee} \tensor K^{i}} 
\]
(where $K^{i}$ and $K^{j}$ are interpreted as $\kappa^{i}$ and $\kappa^{j}$ when $i=n$ or $j=n$, respectively) is defined as follows:
\begin{equation}\label{eq:dnij}
\overline{d(n)}_{ij} = \begin{cases}
\overline{d}_{nj} + \eta_{n}\circ \overline{d}_{n+1,n} \circ \overline{d}_{nj} & \text{if } i=n, \\
\overline{d}_{ij} &   \text{otherwise}
\end{cases} 
\end{equation}
where $\overline{d}_{ij}$ (often written as $\overline{d}_{i,j}$ for clarity) denotes the $(i,j)$-component of $\overline{D}$ in Lemma~\ref{lem:KerPhi} (see also Corollary~\ref{cor:KerPhi}), and $\eta_{i}$ denotes the $\sections{SL^{\vee}}$-linear extension of the bundle map $\eta_{i}$ from Lemma~\ref{lem:U-Exact}, by abuse of notation. 
Here, we let $\overline{d}_{b+1,b}=\eta_{b}=0$.
Note that by Corollary~\ref{cor:KerPhi} and Eqs.~\eqref{eq:Deltas} and~\eqref{eq:Etas}, 
the $(n,j)$-component of $\overline{D(n)}$ satisfies:
\begin{equation}\label{eq:Dn}
\overline{d(n)}_{nj}  = \overline{d}_{nj} + \eta_{n}\circ \overline{d}_{n+1,n} \circ \overline{d}_{nj}
= \overline{d}_{nj}- \eta_{n}\circ \overline{\lambda}_{1} \circ \overline{d}_{nj} 
= \overline{d}_{nj}- \eta_{n}\circ \delta_{n} \circ \overline{d}_{nj} 
= \delta_{n-1}\circ \eta_{n-1}\circ \overline{d}_{nj},
\end{equation}
confirming that $\overline{D(n)}$ is a well-defined degree $+1$ operator on $\cE(n)$. 

The following lemma shows that $\overline{D(n)}$ is in fact a differential and, moreover, describes the structure of the DG module $(\cE(n),\overline{D(n)})$.

\begin{lemma}\label{lem:En}
Let $U\subset M$ be the open neighbourhood of a classical point $p\in M$ constructed in Lemma~\ref{lem:U-Exact}. Then:
\begin{enumerate}
\item For each $n=1,\ldots,b$, the pair $(\cE(n),\overline{D(n)})$ is a DG $(\sections{SL^{\vee}},\lambda^{\transpose})$-module. \label{item:first}
\item For $n=b$, we have $(\cE(b), \overline{D(b)}) \cong \ker \Phi_{\ast}|_{U}$. \label{item:second}
\item For each $n=2,\ldots,b$, there exists an isomorphism of DG $(\sections{SL^{\vee}},\lambda^{\transpose})$-modules:
\[ (\cE(n),\overline{D(n)}) \cong (\cE(n-1),\overline{D(n-1)})\oplus \mathcal{F}(n)\, ,\]
where $\mathcal{F}(n)$ is acyclic. 
In particular, the DG modules $(\cE(n),\overline{D(n)})$ are mutually quasi-isomorphic.
\label{item:third}
\end{enumerate}
\end{lemma}

\begin{proof}
By Lemma~\ref{lem:KerPhi}, assertion~\eqref{item:second} follows once we note that $\overline{d(b)}_{bj}=\overline{d}_{bj}$, 
which follows directly from Eqs.~\eqref{eq:Etas} and~\eqref{eq:Dn}. 
Consequently, assertion~\eqref{item:first} holds for $n=b$.

We now prove assertions~\eqref{item:first} and~\eqref{item:third} by descending induction. 
Let $2\leq n\leq b$ and assume that $(\cE(n),\overline{D(n)})$ is a DG module.
We prove assertion~\eqref{item:third} for this $n$; 
Claim~\ref{claim:1} below will then show that $(\cE(n-1),\overline{D(n-1)})$ is a DG module,
establishing the induction step for assertion~\eqref{item:first}.

Define an automorphism $\alpha:=\id - \sum_{j=0}^{n-2} \alpha_{n-1,j}$ of the graded $\sections{SL^{\vee}}$-module $\cE(n)$ where, for each $0\leq j \leq n-2$, the map
\[\alpha_{n-1,j}:\sections{SL^{\vee}\tensor K^{j}} \to \sections{SL^{\vee}\tensor K^{n-1}}\, ,\]
viewed as an endomorphism of $\cE(n)$,
is the graded $\sections{SL^{\vee}}$-linear map defined by
\[\alpha_{n-1,j}(k_{j})=\eta_{n-1}\circ \overline{d}_{nj}(k_{j})\]
for any $k_{j} \in \sections{K^{j}} \subset \sections{SL^{\vee}\tensor K^{j}}$.
Indeed, the inverse of $\alpha$ is given by $\alpha^{-1}=\id + \sum_{j=0}^{n-2}\alpha_{n-1,j}$.

Define $\overline{D(n)}^{\alpha}:=\alpha\circ \overline{D(n)}\circ \alpha^{-1}$. Then $\alpha$ is an isomorphism of DG $(\sections{SL^{\vee}},\lambda^{\transpose})$-modules
\[ \alpha: (\cE(n),\overline{D(n)}) \to (\cE(n),\overline{D(n)}^{\alpha}) \, .\]

We first show the following:
\begin{claim}\label{claim:1}
 Via the natural inclusion $\cE(n-1)\subset \cE(n)$, the submodule $\cE(n-1)$ is invariant under the differential $\overline{D(n)}^{\alpha}$, and the restriction $\overline{D(n)}^{\alpha}|_{\cE(n-1)}$ agrees with $\overline{D(n-1)}$.
\end{claim}

\begin{proof}[Proof of Claim~\ref{claim:1}]
Note that, by Corollary~\ref{cor:KerPhi} and Eqs.~\eqref{eq:Deltas} and~\eqref{eq:Etas}, we have
\begin{equation}\label{eq:eded}
\eta_{i}\circ \delta_{i}=-\eta_{i}\circ \overline{d}_{i+1,i}, \qquad \delta_{i}\circ \eta_{i}\circ \delta_{i}=\delta_{i}
\end{equation}
for each $i=0,1,\ldots,b-1$.
These identities allow us to compute the $(i,j)$-component $\overline{d(n)}_{ij}^{\alpha}$ of $\overline{D(n)}^{\alpha}$.

First, for $j<n-1$, we have
\begin{align*}
\overline{d(n)}_{nj}^{\alpha} &= \overline{d(n)}_{nj}+\overline{d(n)}_{n,n-1} \circ \alpha_{n-1,j} \\
& = \delta_{n-1}\circ \eta_{n-1}\circ \overline{d}_{nj}+ \delta_{n-1}\circ \eta_{n-1}\circ \overline{d}_{n,n-1}\circ \eta_{n-1}\circ \overline{d}_{nj}\\
&=\delta_{n-1}\circ \eta_{n-1} \circ \overline{d}_{nj} - \delta_{n-1}\circ \eta_{n-1} \circ \delta_{n-1}\circ \eta_{n-1} \circ \overline{d}_{nj}
= 0,
\end{align*}
where the first equality follows from the definition of $\alpha$,
while the remaining equalities follow from Eqs.~\eqref{eq:Dn} and~\eqref{eq:eded}.

Observe that, since $\overline{D}$ is a differential, the relation $\overline{D}^{2}=0$ implies that, for each $i>j$, 
\[
 \sum_{k=j}^{i} \overline{d}_{ik} \circ \overline{d}_{kj}=0 \, .
\]
Moreover, the graded $\sections{SL^{\vee}}$-linearity of $\eta_i$, together with Corollary~\ref{cor:KerPhi}, gives
\[
\overline{d}_{ii}\circ \eta_{i}+\eta_{i}\circ \overline{d}_{i+1,i+1}=0
\]
for each $i=0,1,\ldots,b-1$.
Combining these equations, again for $j<n-1$, we have
\begin{align*}
\overline{d(n)}_{n-1,j}^{\alpha} 
&= \overline{d(n)}_{n-1,j }+ \overline{d(n)}_{n-1,n-1}\circ \alpha_{n-1,j} - \sum_{k=j}^{n-2} \alpha_{n-1,k}\circ \overline{d(n)}_{kj} \\
&= \overline{d}_{n-1,j}+ \overline{d}_{n-1,n-1}\circ \eta_{n-1}\circ \overline{d}_{nj} -\sum_{k=j}^{n-2}\eta_{n-1}\circ \overline{d}_{nk}\circ \overline{d}_{kj}\\
&=\overline{d}_{n-1,j}+ (\overline{d}_{n-1,n-1}\circ \eta_{n-1}\circ \overline{d}_{nj} +\eta_{n-1}\circ \overline{d}_{n,n}\circ \overline{d}_{n,j} )+ \eta_{n-1}\circ \overline{d}_{n,n-1}\circ \overline{d}_{n-1,j}\\
&=\overline{d}_{n-1,j}+\eta_{n-1}\circ \overline{d}_{n,n-1}\circ \overline{d}_{n-1,j}\\
&=\overline{d(n-1)}_{n-1,j},
\end{align*}
where the first equality follows from the definition of $\alpha$ and the lower triangularity of $\overline{D(n)}$, and the last equality follows from Eq.~\eqref{eq:dnij}.
By definition, the image of $\overline{d(n-1)}_{n-1,j}$ lies in $\cE(n-1)$; hence the same is true for the image of $\overline{d(n)}^\alpha_{n-1,j}$.

Finally, for $j=n-1$, Corollary~\ref{cor:KerPhi} shows that 
$\overline{d(n)}_{i,n-1}=0$ for $i<n-1$ and 
\[\overline{d(n)}_{n-1,n-1}=\lambda^{\transpose}\tensor \id:\sections{SL^{\vee}\tensor K^{n-1}} \to \sections{SL^{\vee}\tensor K^{n-1}}\, .\]

By Eqs.~\eqref{eq:Etas}, the operator 
\[\delta_{n-2}\circ \eta_{n-2}:\sections{SL^{\vee}\tensor K^{n-1}}\to \sections{SL^{\vee}\tensor K^{n-1}}\]
is the projection onto $\sections{SL^{\vee}\tensor\kappa^{n-1}}$. 
Together with Eq.~\eqref{eq:Dn}, this gives
\[\overline{d(n-1)}_{n-1,n-1}(\xi\tensor k_{n-1})=\delta_{n-2}\circ \eta_{n-2}(\lambda^{\transpose}(\xi)\tensor k_{n-1})=\lambda^{\transpose}(\xi)\tensor k_{n-1} = \overline{d(n)}_{n-1,n-1}(\xi\tensor k_{n-1})\]
for all $\xi\tensor k_{n-1} \in \sections{SL^{\vee}\tensor\kappa^{n-1}}$.
Moreover, for $k\in\sections{\kappa^{n-1}}$, Eqs.~\eqref{eq:Dn}
and~\eqref{eq:eded} give
\[\overline{d(n)}^\alpha_{n,n-1}(k) =\overline{d(n)}_{n,n-1}(k) =-\delta_{n-1}(k)=0. \]
Hence, $\overline{D(n)}^{\alpha}|_{\cE(n-1)}=\overline{D(n-1)}$. This completes the proof of Claim~\ref{claim:1}.
\end{proof}

It remains to identify the acyclic summand in assertion~\eqref{item:third}.
To do so, for $n=1,\ldots,b$, we use Eqs.~\eqref{eq:Etas} to identify (where $\kappa^{0}=0$)

\[\sections{SL^{\vee}\tensor K^{n-1}}\cong \sections{SL^{\vee}\tensor(\kappa^{n-1}\oplus \kappa^{n}[1])}\, ,\]
where the symbol $[1]$ denotes the degree shift by $1$. 
This yields an identification (where $\cE(0):=0)$
\begin{equation}\label{eq:IdentifycEn}
\cE(n)\cong \cE(n-1)\oplus \sections{SL^{\vee}\tensor(\kappa^{n}[1]\oplus \kappa^{n})},
\end{equation}
which identifies a homogeneous element $\xi\tensor k_{n}[1] \in \sections{SL^{\vee}\tensor\kappa^{n}[1]}$ with an element 
\[\xi\tensor \eta_{n-1}(k_{n})=(-1)^{\degree{\xi}}\eta_{n-1}(\xi\tensor k_{n})\in \sections{SL^{\vee}\tensor K^{n-1}}\subset \cE(n)\, ,\]
where the sign $(-1)^{\degree{\xi}}$ arises from the Koszul rule.

Observe that, when restricted to $\sections{SL^{\vee}\tensor (K^{n-1}\oplus \kappa^{n})}$, 
the differentials $\overline{D(n)}$ and $\overline{D(n)}^\alpha$ agree.
Moreover, under the identification~\eqref{eq:IdentifycEn}, 
Corollary~\ref{cor:KerPhi} implies that
$\sections{SL^{\vee}\tensor(\kappa^{n}[1]\oplus \kappa^{n})}$ is invariant under the differential $\overline{D(n)}$. 
Thus, 
\begin{equation}\label{eq:Fn}
\mathcal{F}(n):= (\sections{SL^{\vee}\tensor(\kappa^{n}[1]\oplus \kappa^{n})}, \overline{D(n)}\big|_{\sections{SL^{\vee}\tensor(\kappa^{n}[1]\oplus \kappa^{n})}})
\end{equation}
is a DG module.

\begin{claim}\label{claim:2}
For each $n=1,\ldots,b$, 
the DG module $\mathcal{F}(n)$ is acyclic.
\end{claim}
\begin{proof}[Proof of Claim~\ref{claim:2}]
By Corollary~\ref{cor:KerPhi}, we have $\overline{d(n)}_{ii}=\lambda^{\transpose}\tensor \id$ for all $i=0,1,\ldots,n$. In particular, when $i=n$, Eq.~\eqref{eq:Dn} implies that the space $\sections{SL^{\vee}\tensor \kappa^{n}}$ is invariant under the map $\lambda^{\transpose}\tensor \id$.

Next, by Eqs.~\eqref{eq:Dn} and~\eqref{eq:eded}, we have
\begin{multline*}
  \overline{d(n)}_{n,n-1}\big(\xi\tensor\eta_{n-1}(k_{n})\big)
= (-1)^{\degree{\xi}} \delta_{n-1}\circ \eta_{n-1}\circ \overline{d}_{n,n-1}\circ \eta_{n-1}(\xi\tensor k_{n})\\
=(-1)^{\degree{\xi}+1} \delta_{n-1}\circ \eta_{n-1}(\xi\tensor k_{n})=(-1)^{\degree{\xi}+1} \xi\tensor k_{n}
\end{multline*}
for homogeneous $\xi\tensor k_{n}\in \sections{SL^{\vee}\tensor \kappa^{n}}$. 
The last equality follows from
$\delta_{n-1}\circ\eta_{n-1}|_{\kappa^n}=\id_{\kappa^n}$,
which follows from Eqs.~\eqref{eq:Etas}.

Under the identification~\eqref{eq:IdentifycEn}, 
the differential $\overline{D(n)}$ restricted to $\sections{SL^{\vee}\tensor (\kappa^{n}[1] \oplus \kappa^{n})}$ is represented by the matrix:
\begin{equation} \label{eq:DiffFn}
\overline{D(n)}\big|_{\sections{SL^{\vee}\tensor(\kappa^{n}[1]\oplus \kappa^{n})}}=
\begin{bmatrix}
\lambda^{\transpose}\tensor \id & 0 \\
\widetilde{\id} & \lambda^{\transpose}\tensor \id
\end{bmatrix}
\end{equation}
where the off-diagonal map $\widetilde{\id}: \sections{SL^{\vee}\tensor \kappa^{n}[1]}\to \sections{SL^{\vee}\tensor \kappa^{n}}$ is defined by
\[\widetilde{\id}(\xi\tensor k_{n}[1])=(-1)^{\degree{\xi}+1}\xi\tensor k_{n}\]
for each homogeneous element $\xi\tensor k_{n} \in \sections{SL^{\vee}\tensor \kappa^{n}}$. Here, the symbol $k[1]$ denotes the element $k$ whose degree is shifted by $1$.

Let $C=\bigl(\sections{SL^{\vee}\tensor\kappa^n}, \lambda^{\transpose}\tensor\id\bigr)$.
Under the degree-shift isomorphism $\sections{SL^{\vee}\tensor\kappa^n[1]}\cong C[1]$,
the differential~\eqref{eq:DiffFn} is the mapping cone differential of the cochain isomorphism $-\id_C:C\to C$. 
Hence $\mathcal F(n)$ is acyclic by~\cite[Corollary~1.5.4]{MR1269324}.
This completes the proof of Claim~\ref{claim:2}.
\end{proof}

Combining Eq.~\eqref{eq:IdentifycEn}, Claim~\ref{claim:1}, and Claim~\ref{claim:2}, 
we obtain assertion~\eqref{item:third} for this $n$, completing the induction step. 
Hence assertions~\eqref{item:first} and~\eqref{item:third} follow, and the proof of the lemma is complete.
\end{proof}

\begin{corollary}\label{cor:Classical}
Let $\Phi=(f,\phi_{1}):(M,L,\lambda)\to (N,E,\mu)$ be an acyclic linear fibration in the category $\cLinfCat$, 
where $f:M\to N$ and $\phi_{1}:L\to f^{\ast}E$. 
Then, for each classical point $p\in M$, there exists an open neighbourhood $U$ of $p$ such that the DG module
$\ker \Phi_{\ast}|_{U}$ as in Eq.~\eqref{eq:KerPhiU} is acyclic.
\end{corollary}

\begin{proof}
By Lemmas~\ref{lem:U-Exact} and~\ref{lem:En}, it suffices to show that
$(\cE(1),\overline{D(1)})$ is acyclic. Under the
identification~\eqref{eq:IdentifycEn}, we have
$(\cE(1),\overline{D(1)})\cong\mathcal F(1)$, which is acyclic by
Claim~\ref{claim:2}.
\end{proof}

We finally prove Proposition~\ref{prop:PsiMap}.
\begin{proof}[Proof of Proposition~\ref{prop:PsiMap}]
By Lemma~\ref{lem:Equiv}, it suffices to show that $\ker \Phi_{\ast}$, which corresponds to $\ker \Psi_{\ast}$, is acyclic. By Corollary~\ref{cor:Locality}, it further suffices to show that, for each $p\in M$, there exists an open neighbourhood $U$ of $p$ such that $\ker \Phi_{\ast}|_{U}$ is acyclic.

Each point $p\in M$ is either classical or non-classical. If $p$ is non-classical, then by Lemma~\ref{lem:NonClassical}, there exists an open neighbourhood $U$ of $p$ such that $\ker \Phi_{\ast}|_{U}$ is acyclic. If $p$ is classical, the same conclusion follows from Corollary~\ref{cor:Classical}.

Therefore, $\ker \Phi_{\ast}$ is acyclic, and hence so is $\ker \Psi_{\ast}$. This completes the proof of Proposition~\ref{prop:PsiMap}.
\end{proof}

\subsection{$(p,q)$-tensor fields and weak equivalences}\label{subsec:TensorBundles}
We show that the analogue of Theorem~\ref{thm:main1} holds for differential $1$-forms and, more generally, for $(p,q)$-tensor fields.

\subsubsection{Differential $1$-forms and weak equivalences}
We begin by stating the theorem for differential $1$-forms.
\begin{theorem} \label{thm:main2}
Let $(\cM,Q)$ and $(\cN,R)$ be DG manifolds of positive amplitude, and let $\Psi:(\cM,Q)\to (\cN,R)$ be a weak equivalence.
Then the cohomology groups of differential $1$-forms $H^{\bullet}(\Omega^{1}(\cM),\liederivative{Q})$ and $H^{\bullet}(\Omega^{1}(\cN),\liederivative{R})$ are isomorphic.
\end{theorem}
 
Dualising the induced map
$\Psi_{\ast}:\tangent{\cM}\to \Psi^{\ast}\tangent{\cN}$ over the same graded manifold $\cM$, 
we obtain the map $\Psi^{\star}:\Psi^{\ast}\cotangent{\cN} \to \cotangent{\cM}$ on the dual bundles. 
Similarly, by dualising the differentials in the diagram~\eqref{eq:PairMorph}, we obtain differentials on $\Omega^{1}(\cM)$ and $\sections{\Psi^{\ast}\cotangent{\cN}}$, denoted by the same symbols $\liederivative{Q}$ and $\liederivative{Q,R}$, respectively. 
Together, we obtain a morphism of DG $(\smooth{\cM},Q)$-modules
 \begin{equation}\label{eq:CotangentPsi}
\Psi^{\star}:(\sections{\Psi^{\ast}\cotangent{\cN}},\liederivative{Q,R}) \to (\Omega^{1}(\cM),\liederivative{Q}).
\end{equation}
 Also, as in the case of vector fields, there is a natural cochain map
\begin{equation}\label{eq:CotangentI}
I:(\Omega^{1}(\cN),\liederivative{R})\to (\sections{ \Psi^{\ast}\cotangent{\cN}},\liederivative{Q,R})
\end{equation}
defined by 
\[I(\omega)=1\tensor \omega \in \smooth{\cM}\tensor_{\smooth{\cN}}\Omega^{1}(\cN)\cong \sections{ \Psi^{\ast}\cotangent{\cN}}\, .\] 

As noted in Remark~\ref{rem:ALF}, an analogous version of Lemma~\ref{lem:ALF} also holds for differential $1$-forms. 
Theorem~\ref{thm:main2} then follows from this fact together with the following theorem.
\begin{theorem}\label{thm:main2-1}
Let $(\cM,Q)$ and $(\cN,R)$ be DG manifolds of positive amplitude.
If $\Psi:(\cM,Q)\to (\cN,R)$ is an acyclic linear fibration,
then both maps \eqref{eq:CotangentPsi} and \eqref{eq:CotangentI} are quasi-isomorphisms.
\end{theorem}
The proof of Theorem~\ref{thm:main2-1} closely follows that of Theorem~\ref{thm:main1-1}; we briefly outline the corresponding steps.

\begin{proposition}\label{prop:CotangentIMap}
Under the hypothesis of Theorem~\ref{thm:main2-1}, the map $I$ in~\eqref{eq:CotangentI} is a quasi-isomorphism.
\end{proposition}
\begin{proof}
Apply Corollary~\ref{cor:KeyLem2} to the quasi-isomorphism $\Psi^{\ast}:(\smooth{\cN},R)\to (\smooth{\cM},Q)$ and the cotangent bundle $(\cotangent{\cN},\cQ_{\cotangent{\cN}})$.
\end{proof}

Using Proposition~\ref{prop:Cotangent} instead of Proposition~\ref{prop:Tangent}, we have an analogue of Lemma~\ref{lem:Equiv}.
\begin{lemma}\label{lem:CotangentEquiv}
Under the hypothesis of Lemma~\ref{lem:Equiv}, the following are equivalent.
\begin{enumerate}
\item \label{item:CE1} The map $\Psi^{\star}$ in~\eqref{eq:CotangentPsi} is a quasi-isomorphism.
\item \label{item:CE2} The map
\[\Phi^{\star}:=
\begin{bmatrix}
f^{\star}  & 0 \\
0 &  \phi_{1}^{\transpose}
\end{bmatrix}
:  (\sections{SL^{\vee}\tensor (f^{\ast}\cotangent{N}\oplus f^{\ast}E^{\vee})}, D_{\mu}^{\transpose}) \to (\sections{SL^{\vee} \tensor (\cotangent{M}\oplus L^{\vee})}, D_{\lambda}^{\transpose})
\]
is a quasi-isomorphism.
\item \label{item:CE3} The DG module $\coker\Phi^{\star}$ is acyclic. 
\end{enumerate}
\end{lemma}
\begin{proof}
The equivalence between~\eqref{item:CE1} and~\eqref{item:CE2} follows by taking the $\sections{SL^\vee}$-linear duals of the identifications used in the proof of Lemma~\ref{lem:Equiv}.

For the equivalence between~\eqref{item:CE2} and~\eqref{item:CE3}, observe that both maps $f^{\star}$ and $\phi_{1}^{\transpose}$ are injective. Thus, the induced map $\Phi^{\star}$ is injective, inducing an exact sequence of cochain complexes
\[ 0\to  (\sections{SL^{\vee}\tensor (f^{\ast}\cotangent{N}\oplus f^{\ast}E^{\vee})}, D_{\mu}^{\transpose})\xto{\Phi^{\star}} (\sections{SL^{\vee} \tensor (\cotangent{M}\oplus L^{\vee})}, D_{\lambda}^{\transpose}) \to \coker \Phi^{\star} \to 0\, .\]
By a standard argument in homological algebra (see for instance \cite[Theorem~1.3.1]{MR1269324}), the cochain complex $\coker \Phi^{\star}$ is acyclic if and only if $\Phi^{\star}$ is a quasi-isomorphism. This proves the equivalence between~\eqref{item:CE2} and~\eqref{item:CE3}.
\end{proof}

Analogously to the case of $\ker \Phi_{\ast}$, the assignment $U\mapsto \coker \Phi^{\star}|_{U}$ defines a sheaf of DG $(\cA,\lambda^{\transpose})$-modules on $M$.
Using Lemma~\ref{lem:Locality}, we obtain an analogue of Corollary~\ref{cor:Locality}. The proof is identical to that of Corollary~\ref{cor:Locality}, with $\ker \Phi_{\ast}$ replaced by $\coker \Phi^{\star}$. This proves the following lemma.
\begin{lemma}
The DG $(\sections{SL^{\vee}},\lambda^{\transpose})$-module $\coker \Phi^{\star}$ is acyclic if, for every point $p\in M$, there exists an open neighbourhood $U$ of $p$ such that the restriction $\coker \Phi^{\star}|_{U}$ is acyclic.
\end{lemma}

For a non-classical point $p\in M$, the analogue of Lemma~\ref{lem:NonClassical} holds, with the same proof after replacing $\ker \Phi_{\ast}$ by $\coker \Phi^{\star}$:
\begin{lemma}
Let $\Phi:(M,L,\lambda)\to (N,E,\mu)$ be a linear fibration in $\cLinfCat$. Then, for $U=M-Z(\lambda_{0})$, the DG module $\coker \Phi^{\star}|_{U}$ is acyclic.
\end{lemma}

Before we move on to the case of classical points, we fix the following notation: 
\[C^{0}:=\coker f^{\star}, \quad \coker \phi_{1}^{\transpose}=\bigoplus_{i=1}^{b}C^{-i},\quad 
\coker \Phi^{\star}=(\sections{SL^{\vee}\tensor (\bigoplus_{i=0}^{b} C^{-i})}, \overline{D}^{\transpose})\, ,\]
where the $(i,j)$-component
\[
\overline d_{ij}^{\transpose}:
\sections{SL^\vee\tensor C^{-i}}
\to
\sections{SL^\vee\tensor C^{-j}}
\]
of $\overline D^{\transpose}$ is the dual of
$\overline d_{ij}$ in Corollary~\ref{cor:KerPhi}.

Now, let $p\in M$ be a classical point. By Lemma~\ref{lem:U-Exact}, we obtain the following lemma.
\begin{lemma}\label{lem:U-Exact2}
Let $\Phi=(f,\phi_{1}):(M,L,\lambda)\to (N,E,\mu)$ be an acyclic linear fibration in the category $\cLinfCat$, where $f:M\to N$ and $\phi_{1}:L\to f^{\ast}E$. Let $\nabla=\nabla^{L}$ be the $\tangent{M}$-connection on $L$ defined as in Eq.~\eqref{eq:NablaL}. 
For each classical point $p\in M$, there exists an open neighbourhood $U$ of $p$ such that
\[
\begin{tikzcd}
 0 \arrow{r}& C^{-b}|_{U} \arrow{r}{\delta_{-b+1}^{\transpose}} & C^{-b+1}|_{U} \arrow{r}& \cdots \arrow{r}& C^{-2}|_{U}\arrow{r}{\delta_{-1}^{\transpose}} & C^{-1}|_{U} \arrow{r}{\delta_{0}^{\transpose}} & C^{0}|_{U} \arrow{r}& 0
 \end{tikzcd}\]
is an exact sequence of vector bundles over $U$, dual to~\eqref{eq:DeltaSES}. 
Here, $\delta_{-i}^{\transpose}$ denotes the dual of $\delta_i$ in Lemma~\ref{lem:U-Exact} for $i=0,\ldots,b-1$; similarly, $\eta_{-i}^{\transpose}$ denotes the dual of $\eta_i$.
\end{lemma}
\begin{proof}
Under the natural identifications $C^0\cong(K^0)^\vee$ and $C^{-i}\cong(K^i)^\vee$, the assertion follows by dualising~\eqref{eq:DeltaSES}.
\end{proof}

We now take the graded $\sections{SL^\vee}$-linear dual of the construction of $(\cE(n),\overline{D(n)})$ in Lemma~\ref{lem:En}.
More precisely, set $\chi^{-b}:=C^{-b}$ and
$\chi^{-i}:=\coker\delta_{-i}^{\transpose} \cong(\kappa^i)^\vee$ for $1\leq i<b$. 
Then, for each $n=1,\ldots,b$, let
\[\cE(-n)^{\transpose}= (\sections{SL^{\vee}\tensor (\bigoplus_{i=0}^{n-1}C^{-i}\oplus \chi^{-n})} )\, , \quad \overline{D(-n)}^{\transpose}=(\overline{d(-n)}^{\transpose}_{ij})_{0\leq i,j \leq n} \, ,\]
where each $\overline{d(-n)}^{\transpose}_{ij}$ is the dual of $\overline{d(n)}_{ij}$ in Eq.~\eqref{eq:dnij}. In this setting, we have an analogue of Lemma~\ref{lem:En}.
\begin{lemma}\label{lem:CotangentEn}
Let $U\subset M$ be the open neighbourhood of a classical point $p\in M$ constructed in Lemma~\ref{lem:U-Exact2}. Then:
\begin{enumerate}
\item \label{item:CEN1} For each $n=1,\ldots,b$, the pair $(\cE(-n)^{\transpose}, \overline{D(-n)}^{\transpose})$ is a DG $(\sections{SL^{\vee}},\lambda^{\transpose})$-module.
\item \label{item:CEN2} For $n=b$, we have $(\cE(-b)^{\transpose}, \overline{D(-b)}^{\transpose})\cong \coker \Phi^{\star}|_{U}$.
\item \label{item:CEN3} For each $n=2,\ldots,b$, there exists an isomorphism of DG $(\sections{SL^{\vee}},\lambda^{\transpose})$-modules:
\[ (\cE(-n)^{\transpose},\overline{D(-n)}^{\transpose}) \cong (\cE(-n+1)^{\transpose},\overline{D(-n+1)}^{\transpose}) \oplus \mathcal{F}(-n)^{\transpose}\, ,\]
where $\mathcal{F}(-n)^{\transpose}$ is acyclic.
\item \label{item:CEN4} For $n=1$, the DG module $(\cE(-1)^{\transpose}, \overline{D(-1)}^{\transpose})$ is acyclic.
\end{enumerate}
\end{lemma}

\begin{proof}
Under the natural identifications
\[
C^{-i}\cong(K^i)^\vee,
\qquad
\chi^{-i}\cong(\kappa^i)^\vee,
\]
the DG module
$(\cE(-n)^{\transpose},\overline{D(-n)}^{\transpose})$
is the graded $\sections{SL^\vee}$-linear dual of
$(\cE(n),\overline{D(n)})$. For $n=b$, this uses the natural
identification
\[
\coker\Phi^\star|_U\cong(\ker\Phi_\ast|_U)^\vee.
\]
Thus, assertions~\eqref{item:CEN1} and~\eqref{item:CEN2}, as well as the decomposition in assertion~\eqref{item:CEN3}, follow by dualising Lemma~\ref{lem:En}, 
where $\mathcal F(-n)^{\transpose}:=\mathcal F(n)^\vee$.
Since $\mathcal F(n)$ is the mapping cone of a cochain isomorphism, it
is contractible, and hence so is its dual. This completes the proof of assertion~\eqref{item:CEN3}.
Finally, assertion~\eqref{item:CEN4} follows by dualising the identification 
$(\cE(1),\overline{D(1)})\cong\mathcal F(1)$.
\end{proof}

Combining Lemmas~\ref{lem:CotangentEquiv}--\ref{lem:CotangentEn}, we have proved the following:
\begin{proposition}\label{prop:CotangentPsiMap}
Under the hypothesis of Theorem~\ref{thm:main2-1}, the map $\Psi^{\star}$ in~\eqref{eq:CotangentPsi} is a quasi-isomorphism.
\end{proposition}

Finally, we prove Theorem~\ref{thm:main2-1}.
\begin{proof}[Proof of Theorem~\ref{thm:main2-1}]
Theorem~\ref{thm:main2-1} follows from Propositions~\ref{prop:CotangentIMap} and~\ref{prop:CotangentPsiMap}.
\end{proof}

\subsubsection{$(p,q)$-tensor fields and weak equivalences}
We now consider general $(p,q)$-tensor fields. 
Combining the tangent bundle and cotangent bundle cases with Corollary~\ref{cor:KeyLem2}, we obtain a similar result for this case.

\begin{theorem}\label{thm:main3}
Let $(\cM,Q)$ and $(\cN,R)$ be DG manifolds of positive amplitude, and let $\Psi:(\cM,Q)\to (\cN,R)$ be a weak equivalence.
For all nonnegative integers $p,q$, the cohomology groups of $(p,q)$-tensor fields $H^{\bullet}(\sections{\cotangent{\cM}^{\tensor q}\tensor \tangent{\cM}^{\tensor p}},\liederivative{Q})$ and $H^{\bullet}(\sections{\cotangent{\cN}^{\tensor q}\tensor \tangent{\cN}^{\tensor p}},\liederivative{R})$ are isomorphic.
\end{theorem}

\begin{proof}
Again, as mentioned in Remark~\ref{rem:ALF}, we may assume that $\Psi$ is an acyclic linear fibration.

Following the proofs of Proposition~\ref{prop:IMap} and Proposition~\ref{prop:CotangentIMap}, the map
\[ I : (\sections{\cotangent{\cN}^{\tensor q}\tensor \tangent{\cN}^{\tensor p}},\liederivative{R}) \to (\sections{\Psi^{\ast}(\cotangent{\cN}^{\tensor q}\tensor\tangent{\cN}^{\tensor p})},\liederivative{Q,R})\]
is a quasi-isomorphism. Thus, it remains to show that the cohomology groups $H^{\bullet}(\sections{\cotangent{\cM}^{\tensor q}\tensor \tangent{\cM}^{\tensor p}},\liederivative{Q})$ and $H^{\bullet}(\sections{\Psi^{\ast}(\cotangent{\cN}^{\tensor q}\tensor \tangent{\cN}^{\tensor p})},\liederivative{Q,R})$ are isomorphic. We show this by constructing a pair of quasi-isomorphisms connecting their underlying cochain complexes.

By Proposition~\ref{prop:PsiMap}, we have a quasi-isomorphism of DG $(\smooth{\cM},Q)$-modules
\[\Psi_{\ast}:(\XX(\cM),\liederivative{Q})\to (\sections{\Psi^{\ast}\tangent{\cN}},\liederivative{Q,R})\, .\]
For $p\geq 1$, applying Corollary~\ref{cor:KeyLem2} to this map and to the DG module $(\sections{\cotangent{\cM}^{\tensor q}\tensor \tangent{\cM}^{\tensor p-1}},\liederivative{Q})$, we obtain a quasi-isomorphism of DG $(\smooth{\cM},Q)$-modules
\[ \id \tensor \Psi_{\ast} : (\sections{\cotangent{\cM}^{\tensor q} \tensor \tangent{\cM}^{\tensor p}},\liederivative{Q})\to (\sections{\cotangent{\cM}^{\tensor q}\tensor \tangent{\cM}^{\tensor p-1}}\tensor_{\cR}\sections{\Psi^{\ast}\tangent{\cN}},\liederivative{Q}\tensor \id + \id \tensor \liederivative{Q,R})\, ,\]
where $\cR=\smooth{\cM}$. 
Applying Corollary~\ref{cor:KeyLem2} iteratively in this manner---$p$ times---we obtain a quasi-isomorphism of DG $(\smooth{\cM},Q)$-modules:
\[ \id \tensor (\Psi_{\ast})^{\tensor p}: (\sections{\cotangent{\cM}^{\tensor q}\tensor \tangent{\cM}^{\tensor p}},\liederivative{Q}) \to  (\sections{\cotangent{\cM}^{\tensor q}}\tensor_{\cR}\sections{ \Psi^{\ast}\tangent{\cN}^{\tensor p}}, \liederivative{Q}\tensor \id + \id \tensor \liederivative{Q,R})\, .\]
Similarly, applying Corollary~\ref{cor:KeyLem2} iteratively $q$ times to $\Psi^{\star}$ yields a quasi-isomorphism in the
opposite direction.
Combining these quasi-isomorphisms with $I$, we obtain a pair of quasi-isomorphisms $\alpha$ and $\beta$ for all nonnegative integers $p,q$:
\begin{equation}\label{eq:pqDiagram1}
\begin{tikzcd}[column sep=-5em]
(\sections{\cotangent{\cM}^{\tensor q}\tensor \tangent{\cM}^{\tensor p}},\liederivative{Q}) \arrow{rd}{\alpha} 
&& \arrow[swap]{ld}{\beta}(\sections{\cotangent{\cN}^{\tensor q}\tensor \tangent{\cN}^{\tensor p}},\liederivative{R}) \\
& (\sections{\cotangent{\cM}^{\tensor q}}\tensor_{\cR}\sections{ \Psi^{\ast}\tangent{\cN}^{\tensor p}}, \liederivative{Q}\tensor \id + \id \tensor \liederivative{Q,R})&
 \end{tikzcd}
 \end{equation}
This completes the proof of Theorem~\ref{thm:main3}.
\end{proof}

By the proof of Theorem~\ref{thm:main3}, the quasi-isomorphisms in the diagram~\eqref{eq:pqDiagram1} are explicitly written as follows. 
\begin{corollary}\label{cor:main3}
Under the hypotheses of Theorem~\ref{thm:main3}, suppose in addition that $\Psi$ is an acyclic linear fibration.
The two quasi-isomorphisms $\alpha$ and $\beta$ in the diagram~\eqref{eq:pqDiagram1} are given explicitly as follows: let $X_{1},\ldots,X_{q}\in \XX(\cM)$.
Under the identifications $\sections{\cotangent{\cM}^{\tensor q} \tensor \tangent{\cM}^{\tensor p}}\cong \sections{\Hom(\tangent{\cM}^{\tensor q}, \tangent{\cM}^{\tensor p})}$, 
the map $\alpha$ is defined by
\[\alpha(F)(X_{1},\ldots, X_{q})=(\Psi_{\ast})^{\tensor p} \big(F(X_{1},\ldots,X_{q})\big) \]
for $F\in \sections{\Hom(\tangent{\cM}^{\tensor q},\tangent{\cM}^{\tensor p})}$.
The map $\beta$ is the composition of quasi-isomorphisms
\[ \beta: \sections{\cN; \cotangent{\cN}^{\tensor q} \tensor \tangent{\cN}^{\tensor p}} \to \sections{\cM; \Psi^{\ast}(\cotangent{\cN}^{\tensor q} \tensor \tangent{\cN}^{\tensor p})} \to \sections{\cM; \cotangent{\cM}^{\tensor q} \tensor \Psi^{\ast}\tangent{\cN}^{\tensor p}}\]
defined by
\[\beta(G)(X_{1},\ldots,X_{q})=\widetilde{G}(\Psi_{\ast}(X_{1}),\ldots, \Psi_{\ast}(X_{q}))\]
for $G\in \sections{\Hom(\tangent{\cN}^{\tensor q},\tangent{\cN}^{\tensor p})}$, where $\widetilde{G}$ denotes the $\smooth{\cM}$-linear extension of $G$.
\end{corollary}

\begin{remark}\label{rem:Sym}
In Theorem~\ref{thm:main3}, the tensor products can be replaced by the graded symmetric and exterior powers. 
For instance, for any nonnegative integers $p,q$, the cochain complexes $(\sections{S^{q}(\cotangent{\cM})\tensor \Lambda^{p}\tangent{\cM}},\liederivative{Q})$ and $(\sections{S^{q}(\cotangent{\cN})\tensor (\Lambda^{p}\tangent{\cN})},\liederivative{R})$ are quasi-isomorphic under the assumption of Theorem~\ref{thm:main3}. 
Indeed, both $\alpha$ and $\beta$ commute with graded symmetrisation and antisymmetrisation, and therefore restrict to quasi-isomorphisms on the corresponding tensor products.
\end{remark}

\section{Atiyah class, Todd class, and Hochschild cohomology}

\subsection{Atiyah class and Todd class}
We briefly review the definitions of the Atiyah and Todd classes of DG manifolds using affine connections~\cite{MR3319134}.

Let $(\cM,Q)$ be a DG manifold. The graded space of $(1,2)$-tensor fields $\sections{\cotangent{\cM}^{\tensor 2} \tensor \tangent{\cM}}$ on $\cM$ is equipped with the differential $\liederivative{Q}$ defined as the Lie derivative along the homological vector field $Q$. Under the identification 
\[\sections{\cotangent{\cM}^{\tensor 2} \tensor \tangent{\cM}} \cong \sections{\Hom(\tangent{\cM}^{\tensor 2}, \tangent{\cM})}\, ,\]
the differential $\liederivative{Q}$ is defined by
\[\liederivative{Q}F(X,Y)=[Q, F(X,Y)] - (-1)^{\degree{F}}F([Q,X],Y) - (-1)^{\degree{F}+\degree{X}} F(X,[Q,Y])\]
for homogeneous $F\in \sections{\Hom(\tangent{\cM}^{\tensor 2}, \tangent{\cM})}$ and $X,Y\in \sections{\tangent{\cM}}$. The pair $(\sections{\cotangent{\cM}^{\tensor 2} \tensor \tangent{\cM}}, \liederivative{Q})$ forms a DG $(\smooth{\cM},Q)$-module, and in particular, a cochain complex.

Recall that an affine connection $\nabla$ on $\cM$ is a $\RR$-bilinear map
\[\nabla : \sections{\tangent{\cM}}\times \sections{\tangent{\cM}} \to \sections{\tangent{\cM}}\]
of degree $0$, satisfying
\[\nabla_{f\cdot X}Y=f\cdot \nabla_{X}Y, \qquad \nabla_{X} (f\cdot Y) = X(f) \cdot Y + (-1)^{\degree{X}\cdot \degree{f}} f\cdot \nabla_{X}Y \]
for homogeneous $f\in \smooth{\cM}$ and $X,Y\in \sections{\tangent{\cM}}$. 

Define the $\RR$-bilinear map $\liederivative{Q}\nabla : \sections{\tangent{\cM}}\times \sections{\tangent{\cM}} \to \sections{\tangent{\cM}}$
by
\[ (\liederivative{Q}\nabla)(X,Y)=[Q,\nabla_{X}Y] - \nabla_{[Q,X]}Y -(-1)^{\degree{X}}\nabla_{X} [Q,Y]\]
for homogeneous $X,Y\in \sections{\tangent{\cM}}$. 

\begin{proposition}[\cite{MR3319134}]
Given a DG manifold $(\cM,Q)$ and an affine connection $\nabla$ on $\cM$, the following hold.
\begin{enumerate}
\item The element $\liederivative{Q}\nabla$ is a $(1,2)$-tensor field of degree $+1$.
\item The element $\liederivative{Q}\nabla$ is a $1$-cocycle in $(\sections{\cotangent{\cM}^{\tensor 2} \tensor \tangent{\cM}}, \liederivative{Q})$.
\item The cohomology class of $\liederivative{Q}\nabla$ is independent of the choice of $\nabla$.
\end{enumerate}
\end{proposition}

The element $\At^{\nabla} := \liederivative{Q}\nabla$ is called the Atiyah $1$-cocycle associated with the affine connection $\nabla$. The cohomology class 
\[\alpha_{(\cM,Q)}:=[\At^{\nabla}] \in H^{1}\big( \sections{\cotangent{\cM}^{\tensor 2} \tensor \tangent{\cM}}, \liederivative{Q} \big)\]
is called the \textbf{Atiyah class} of the DG manifold $(\cM,Q)$. 
The Atiyah class of the DG manifold $(\cM,Q)$ is an obstruction to the existence of an affine connection $\nabla$ on $\cM$ compatible with the homological vector field $Q$.

By identifying $\sections{\cotangent{\cM}^{\tensor 2} \tensor \tangent{\cM}} \cong \Omega_{\cM}^{1}(\End(\tangent{\cM}))$, we define the $k$-th power of the Atiyah cocycle $(\At^{\nabla})^{k}\in \Omega_{\cM}^{k}(\End(\tangent{\cM}))$ by the $k$-th composition of $\At^{\nabla}$, where $\Omega_{\cM}^{k}$ denotes the graded space of differential $k$-forms on $\cM$. 
The element 
\[\Td^{\nabla}:=\Ber\Bigg(\frac{\At^{\nabla}}{1-\exp(-\At^{\nabla})}\Bigg)\in \prod_{k\geq 0}(\Omega_{\cM}^{k})^{k}\]
is componentwise $\liederivative{Q}$-closed in the product cochain complex $\prod_{k\geq 0}((\Omega_{\cM}^{k})^{\bullet},\liederivative{Q})$, and is called
the Todd cocycle associated with $\nabla$.
Here, $\Ber$ denotes the Berezinian~\cite{MR1632008}, and $(\Omega_{\cM}^{k})^{k}$ denotes the degree $k$-component of differential $k$-forms on $\cM$. The \textbf{Todd class} of $(\cM,Q)$ is the cohomology class
\[\td_{(\cM,Q)}=[\Td^{\nabla}] \in \prod_{k\geq 0} H^{k}((\Omega_{\cM}^{k})^{\bullet},\liederivative{Q})\, ,\]
which is independent of the choice of the connection $\nabla$.
Moreover, the Todd class $\td_{(\cM,Q)}$ admits the following expression in terms of the Atiyah class $\alpha_{(\cM,Q)}$:
\[ \td_{(\cM,Q)}=\Ber\Bigg(\frac{\alpha_{(\cM,Q)}}{1-\exp(-\alpha_{(\cM,Q)})}\Bigg) \, .\]

Our main theorem concerns the behaviour of the Atiyah class and the Todd class of DG manifolds of positive amplitude under weak equivalences.
Recall that, by Theorem~\ref{thm:main3}, a weak equivalence $\Psi:(\cM,Q)\to (\cN,R)$ between DG manifolds of positive amplitude induces an isomorphism 
\begin{equation}\label{eq:H12}
 H^{\bullet}(\sections{\cotangent{\cM}^{\tensor 2}\tensor \tangent{\cM}}, \liederivative{Q}) \cong H^{\bullet}(\sections{\cotangent{\cN}^{\tensor 2}\tensor \tangent{\cN}},\liederivative{R})
 \end{equation}
between cohomology groups of $(1,2)$-tensor fields on $(\cM,Q)$ and $(\cN,R)$. 

The following is the main result of this paper.
\begin{theorem}\label{thm:main4}
Let $(\cM,Q)$ and $(\cN,R)$ be DG manifolds of positive amplitude, and let $\Psi:(\cM,Q)\to (\cN,R)$ be a weak equivalence.
Then the Atiyah class $\alpha_{(\cM,Q)}$ is identified with $\alpha_{(\cN,R)}$ under the isomorphism~\eqref{eq:H12} on cohomology induced by $\Psi$.
\end{theorem}

\begin{proof}
Again, as mentioned in Remark~\ref{rem:ALF}, we may assume that $\Psi$ is an acyclic linear fibration.

The isomorphism~\eqref{eq:H12} is represented by a pair of quasi-isomorphisms 
$\alpha$ and $\beta$ in Corollary~\ref{cor:main3} for $(p,q)=(1,2)$. 
We claim that there exist an affine connection $\nabla^{\cM}$ on $\cM$ and an affine connection $\nabla^{\cN}$ on $\cN$ 
such that their associated Atiyah cocycles agree under these maps $\alpha$ and $\beta$. That is, 
$\alpha(\At^{\nabla^{\cM}})=\beta(\At^{\nabla^{\cN}})$.

Choose any affine connection $\nabla^{\cN}$ on $\cN$. 
Under the identification $\sections{\Psi^{\ast}\tangent{\cN}} \cong \smooth{\cM}\tensor_{\smooth{\cN}}\XX(\cN)$, the pullback connection of $\nabla^{\cN}$ along $\Psi$ is the $\tangent{\cM}$-connection $\widetilde{\nabla}^{\cN}$ on $\Psi^{\ast}\tangent{\cN}$, defined by
\[ \widetilde{\nabla}^{\cN}_{X} (f\tensor Y )=X(f)\tensor Y  + (-1)^{\degree{f}\cdot \degree{X}} f\cdot \nabla^{\cN}_{\Psi_{\ast}X}Y \]
for homogeneous $X\in \XX(\cM)$, $Y\in \XX(\cN)$ and $f\in \smooth{\cM}$. 
Here, $\nabla^{\cN}$ is extended $\smooth{\cM}$-linearly in its first argument to $\sections{\Psi^{\ast}\tangent{\cN}}$.
Then, it is straightforward to check that 
\begin{equation}\label{eq:BetaAt}
\beta(\At^{\nabla^{\cN}})(X_{1},X_{2})=\liederivative{Q,R}\big( \widetilde{\nabla}^{\cN}_{X_{1}} (\Psi_{\ast}X_{2})\big) - \widetilde{\nabla}^{\cN}_{\liederivative{Q}(X_{1})}(\Psi_{\ast}X_{2}) -(-1)^{\degree{X_{1}}}\widetilde{\nabla}^{\cN}_{X_{1}}(\liederivative{Q,R}(\Psi_{\ast}X_{2}))
\end{equation}
for homogeneous $X_{1},X_{2}\in \XX(\cM)$.

By Remark~\ref{rem:KerPsi}, we obtain a short exact sequence of graded vector bundles over $\cM$:
\[
\begin{tikzcd}
 0 \arrow{r}& \ker \Psi_{\ast} \arrow{r}& \tangent{\cM} \arrow{r}{\Psi_{\ast}}& \Psi^{\ast}\tangent{\cN} \arrow{r}& 0,
 \end{tikzcd}
 \]
and each choice of splitting yields an isomorphism
\begin{equation}\label{eq:TMSplit}
 \tangent{\cM}\cong \ker \Psi_{\ast}\oplus \Psi^{\ast}\tangent{\cN}.
 \end{equation}

Choose a $\tangent{\cM}$-connection $\nabla^{\mathcal{K}}$ on $\ker \Psi_{\ast}$. 
Under the above isomorphism, define $\nabla^{\cM}:=\nabla^{\mathcal{K}}+\widetilde{\nabla}^{\cN}$. 
By construction, we have $\Psi_{\ast}\big(\nabla^{\cM}_{X}Y\big) = \widetilde{\nabla}^{\cN}_{X} (\Psi_{\ast}Y)$.
Applying the differentials to this identity and using that $\Psi_{\ast}$ is a cochain map gives
\[\alpha(\At^{\nabla^{\cM}}) =\beta(\At^{\nabla^{\cN}})\, . \]
This completes the proof.
\end{proof}

The analogous result holds for the Todd class.
\begin{theorem}\label{thm:main5}
Under the hypotheses of Theorem~\ref{thm:main4}, the Todd class $\td_{(\cM,Q)}$ is identified with $\td_{(\cN,R)}$ under the isomorphism of cohomology induced by the weak equivalence $\Psi$.
\end{theorem}
\begin{proof}
As above, we may assume that $\Psi$ is an acyclic linear fibration.
Let $\mathcal K=\ker\Psi_\ast$, and use the connections chosen in the
proof of Theorem~\ref{thm:main4}. With respect to the
splitting~\eqref{eq:TMSplit}, the Atiyah cocycle
$\At^{\nabla^{\cM}}$ is block triangular, with diagonal blocks
\[\liederivative{Q}\nabla^{\mathcal K}\in \Omega^{1}_{\cM}(\End(\ker\Psi_{\ast})) \qquad \text{and} \qquad
\Psi^\ast\At^{\nabla^{\cN}}\in \Omega^{1}_{\cM}(\End(\Psi^{\ast}\tangent{\cN})) \, .\]

Since $\ker\Psi_{\ast}$ is acyclic, by Corollary~\ref{cor:KeyLem2}, the cochain complex $(\Omega_{\cM}^{1}(\End(\ker\Psi_{\ast})),\liederivative{Q})$ is acyclic. 
Thus, $\liederivative{Q}\nabla^{\mathcal{K}}$ is a coboundary. 
Therefore, we have $\Td^{\nabla^{\cM}}=\Td^{\mathcal K}\cdot\Psi^{\ast}\Td^{\nabla^{\cN}}$,
where $\Td^{\mathcal{K}}$ denotes the Todd factor associated with $\liederivative{Q}\nabla^{\mathcal{K}}$; since this block is a coboundary, $[\Td^{\mathcal K}]=1$.
As a consequence, we have 
\[ [\alpha(\Td^{\nabla^{\cM}})]=[\alpha\big(\Psi^{\ast}\Td^{\nabla^{\cN}}\big)]=[\beta(\Td^{\nabla^{\cN}})] \, . \]
This completes the proof.
\end{proof}

\begin{remark}\label{rem:Str}
The Todd cocycle $\Td^{\nabla}$ associated with $\nabla$ can be described in terms of the scalar Atiyah $k$-cocycles $c_{k}=\frac{1}{k!} \str \big((\At^{\nabla})^{k}\big) \in (\Omega_{\cM}^{k})^{k}$ in the cochain complex $((\Omega_{\cM}^{k})^{\bullet},\liederivative{Q})$, where $\str$ denotes the supertrace~\cite{MR3319134}. 
It follows from the identity $\Ber(\exp(A))=\exp(\str(A))$ for a supermatrix $A$ that
\[\Td^{\nabla}=\exp\Big(-\sum_{k=1}^{\infty}\frac{B_{k}}{k\cdot k!} \str\big((\At^{\nabla})^{k}\big)\Big)\, ,\]
where $B_{k}$ are Bernoulli numbers, with $B_{1}=-1/2$.

In this context, Theorem~\ref{thm:main5} can alternatively be proved by showing the invariance of scalar Atiyah classes under weak equivalences.
Indeed, for $(p,q)=(1,1)$, the comparison isomorphism on cohomology induced by Corollary~\ref{cor:main3} is compatible with tensor products and the supertrace $\str$.
Combined with Corollary~\ref{cor:KeyLem2} and Remark~\ref{rem:Sym} for $k$-forms, this implies that the scalar Atiyah classes are identified.
\end{remark}

\begin{remark}\label{rem:Td1}
There is a simpler proof of Theorem~\ref{thm:main5}. 
Indeed, since $\cM$ has positive amplitude, $\Omega_{\cM}^k$ is concentrated in nonpositive degrees; hence $(\Omega_{\cM}^k)^k=0$ for $k\geq1$.
Therefore, all the scalar Atiyah classes vanish, and in particular, $\td_{(\cM,Q)}=1$.
Hence, from this viewpoint, the invariance of the Todd class is automatic. 
However, the Atiyah class need not vanish, even for a DG manifold of positive amplitude.
\end{remark}

\subsection{Hochschild cohomology}

Following~\cite{MR2062626}, the Hochschild cohomology of a smooth manifold $M$ is defined as the cohomology of the cochain complex of poly-differential operators on $M$.
Likewise, we define the Hochschild cohomology of a DG manifold $(\cM,Q)$ using the cochain complex induced by poly-differential operators on $(\cM,Q)$. 
This version of Hochschild cohomology is the one appearing in the Kontsevich formality theorem and in the Duflo--Kontsevich-type theorem for DG manifolds~\cite{MR3754617, arXiv:2608.07362}.
Below, we briefly describe the construction.

Given a DG manifold $(\cM,Q)$, denote the graded space of differential operators on $\cM$ by $\cD(\cM)$. The homological vector field $Q$ naturally induces a differential $\llbracket Q, \argument \rrbracket$ on $\cD(\cM)$ defined by the graded commutator with $Q$. The pair $(\cD(\cM),\llbracket Q,\argument \rrbracket)$ forms a DG $(\smooth{\cM},Q)$-module, and in particular, a cochain complex. 

For each $p\geq 0$, we denote by $(\pDpoly^{p}(\cM), \llbracket Q,\argument \rrbracket)$ the DG $(\smooth{\cM},Q)$-module obtained as the tensor product of $p$ copies of $\cD(\cM)$ over $\smooth{\cM}$. 
We denote the degree $q$ component of $\pDpoly^{p}(\cM)$ by $\big(\pDpoly^{p}(\cM)\big)^{q}$.
Note that the space $\pDpoly^{n}(\cM)$ can be viewed as a subspace of $\Hom(\smooth{\cM}^{\tensor n}, \smooth{\cM})$. 
From this viewpoint, the Hochschild differential 
\[\hochschild: \big(\pDpoly^{n}(\cM)\big)^{\bullet} \to \big(\pDpoly^{n+1}(\cM)\big)^{\bullet}\]
is compatible with $\llbracket Q,\argument \rrbracket$. Thus, we obtain a cochain complex of cochain complexes,
\[0\to (\smooth{\cM},Q) \xto{\hochschild} (\cD(\cM), \llbracket Q,\argument \rrbracket) \xto{\hochschild} \cdots \xto{\hochschild}(\pDpoly^{n}(\cM), \llbracket Q,\argument \rrbracket)\xto{\hochschild} \cdots \, , \]
which forms a double complex.

Totalising the double complex by taking direct sum along the diagonal, we obtain a total complex
\[\Dpoly^{n}(\cM)=\bigoplus_{n=p+q} \big(\pDpoly^{p}(\cM)\big)^{q}\]
whose differential is given by $\hochschild + \llbracket Q,\argument \rrbracket$: $\hochschild$ increases $p$ and $\llbracket Q,\argument \rrbracket$ increases $q$. 
We define the Hochschild cohomology of $(\cM,Q)$ to be the cohomology of $(\Dpoly^{\bullet}(\cM),\hochschild + \llbracket Q,\argument \rrbracket)$:
\[
 \grHH^{\bullet}(\cM,Q):=H^{\bullet}(\Dpoly^{\bullet}(\cM), \hochschild + \llbracket Q,\argument \rrbracket)\, .
\]
The Hochschild cohomology $\grHH^{\bullet}(\cM,Q)$ carries a natural Gerstenhaber algebra structure: the product is the cup product $\cup$, and the Lie bracket is the standard Gerstenhaber bracket $\llbracket\argument,\argument\rrbracket$, 
whose restriction to $\cD(\cM)$ is the graded commutator.

\begin{remark}
As above, we define the Hochschild cohomology of $(\cM,Q)$ using the direct sum totalisation of the double complex $(\big(\pDpoly^{p}(\cM)\big)^{q}, \hochschild, \llbracket Q, \argument \rrbracket)$. This construction fits into the framework of the Duflo--Kontsevich-type theorem for DG manifolds~\cite{MR3754617, arXiv:2608.07362}.
However, one may alternatively consider the \emph{direct product} totalisation, instead of the direct sum totalisation. In general, the cohomology groups obtained from these two totalisations are different~\cite{MR4584414}; see also~\cite{MR2931331}.
Moreover, an analogue of the Duflo--Kontsevich-type theorem for DG manifolds does not hold when working with the direct product totalisation~\cite{MR2202177}.
\end{remark}

Our goal in this section is to prove that the Hochschild cohomology $\grHH(\cM,Q)$ of DG manifolds of positive amplitude $(\cM,Q)$ is invariant under weak equivalences. To do so, we use the cochain complex of poly-vector fields and the Hochschild--Kostant--Rosenberg (HKR) theorem for DG manifolds~\cite[Proposition~4.1]{MR3754617} (see also~\cite[Lemma~A.2]{MR2304327}). 

As in the case of poly-differential operators, we define the cochain complex of poly-vector fields by
\[ \Tpoly^{n}(\cM):=\bigoplus_{n=p+q}\big(\sections{\Lambda^{p} \tangent{\cM}}\big)^{q}\]
with differential $\liederivative{Q}$, the Lie derivative along $Q$.
Here, $\big(\sections{\Lambda^{p} \tangent{\cM}}\big)^{q}$ denotes the degree $q$ component of $p$-vector fields, and for each $p\geq 0$, the differential $\liederivative{Q}$ increases the internal degree $q$ of $\sections{\Lambda^{p}\tangent{\cM}}$.
The cohomology group $H^{\bullet}(\Tpoly^{\bullet}(\cM), \liederivative{Q})$ is also endowed with a Gerstenhaber algebra structure: the product is the wedge product $\wedge$, and the Lie bracket is the Schouten--Nijenhuis bracket $[\argument, \argument]$.

The HKR theorem asserts that the cochain complexes of poly-vector fields and poly-differential operators are quasi-isomorphic. More precisely, the HKR map
\[\hkr : (\Tpoly^{\bullet}(\cM), \liederivative{Q}) \to (\Dpoly^{\bullet}(\cM),\hochschild + \llbracket Q,\argument \rrbracket)\, ,\]
defined by $\hkr(f)=f$ for $f\in \smooth{\cM}$, and for $k\geq 1$ and homogeneous $X_{1},\ldots, X_{k}\in \XX(\cM)$,
\[ \hkr(X_{1}\wedge \cdots \wedge X_{k})=\frac{1}{k!} \sum_{\sigma\in S_{k}} \kappa(\sigma) X_{\sigma(1)}\tensor \cdots \tensor X_{\sigma(k)}\]
is a quasi-isomorphism.
Here, $S_{k}$ denotes the symmetric group of degree $k$, and $\kappa(\sigma)=\pm 1$ denotes the sign satisfying
$ X_{1}\wedge\cdots \wedge X_{k}=\kappa(\sigma) X_{\sigma(1)}\wedge \cdots \wedge X_{\sigma(k)}$. 

In the positive-amplitude setting, we obtain the following lemma.
\begin{lemma}\label{lem:TQI}
Let $\Psi:(\cM,Q)\to (\cN,R)$ be a weak equivalence between DG manifolds of positive amplitude. Then the cohomology groups $H^{\bullet}(\Tpoly^{\bullet}(\cM), \liederivative{Q})$ and $H^{\bullet}(\Tpoly^{\bullet}(\cN), \liederivative{R})$ of poly-vector fields are isomorphic.
\end{lemma}
\begin{proof}
As in Remark~\ref{rem:ALF}, it suffices to assume that $\Psi$ is an acyclic linear fibration. 
By Corollary~\ref{cor:main3} and Remark~\ref{rem:Sym}, for each $p\geq0$, there is a pair of quasi-isomorphisms
\[
\begin{tikzcd}
(\sections{\Lambda^{p}\tangent{\cM}}, \liederivative{Q}) \arrow{r}{\alpha} &(\sections{\Lambda^{p}\Psi^{\ast}\tangent{\cN}},\liederivative{Q,R})& \arrow[swap]{l}{\beta} (\sections{\Lambda^{p}\tangent{\cN}},\liederivative{R}) \, .
\end{tikzcd}
\]
Observe that, for $(\cM,Q)$, we have
\[H^{\bullet}(\Tpoly^{\bullet}(\cM), \liederivative{Q})\cong \bigoplus_{p+q=\bullet}H^{q}(\sections{\Lambda^{p}\tangent{\cM}}, \liederivative{Q}) \, , \] 
and similarly for $(\cN,R)$. Therefore the cohomology groups of poly-vector fields on $(\cM,Q)$ and $(\cN,R)$ are isomorphic.
\end{proof}

Denote by $\Psi_{T}: H^{\bullet}(\Tpoly^{\bullet}(\cM), \liederivative{Q}) \to H^{\bullet}(\Tpoly^{\bullet}(\cN), \liederivative{R})$ the isomorphism of graded vector spaces induced by Lemma~\ref{lem:TQI}.
Together with the HKR theorem, one obtains an isomorphism of graded vector spaces
\begin{equation}\label{eq:PsiD}
\Psi_{D}:\grHH(\cM,Q)\to \grHH(\cN,R)
\end{equation}
defined by the following commutative diagram:
\begin{equation}\label{diag:HKR}
\begin{tikzcd}
H^{\bullet}(\Tpoly^{\bullet}(\cM), \liederivative{Q}) \arrow{d}{\Psi_{T}} \arrow{r}{\hkr} & H^{\bullet}(\Dpoly^{\bullet}(\cM),\hochschild + \llbracket Q,\argument \rrbracket) \arrow[dashed]{d}{\Psi_{D}} =\grHH^{\bullet}(\cM,Q)  \\
H^{\bullet}(\Tpoly^{\bullet}(\cN), \liederivative{R}) \arrow{r}{\hkr} & H^{\bullet}(\Dpoly^{\bullet}(\cN),\hochschild + \llbracket R,\argument \rrbracket) = \grHH^{\bullet}(\cN,R) 
\end{tikzcd}
\end{equation}

We show that $\Psi_{D}$ is an isomorphism of Gerstenhaber algebras.

\begin{theorem}\label{thm:maingrHH1}
Let $\Psi:(\cM,Q)\to (\cN,R)$ be a weak equivalence between DG manifolds of positive amplitude. Then $\Psi_{D}$ defined in~\eqref{eq:PsiD} is an isomorphism of Gerstenhaber algebras.
\end{theorem}

\begin{proof}
By Remark~\ref{rem:Td1}, the Todd class of a DG manifold of positive amplitude is the multiplicative identity $1$. Therefore, by the Duflo--Kontsevich-type theorem for DG manifolds~\cite[Theorem~4.3]{MR3754617}, the map induced by $\hkr$ on cohomology is an isomorphism of Gerstenhaber algebras.

As in Remark~\ref{rem:ALF}, it suffices to assume that $\Psi$ is an acyclic linear fibration.
Using the explicit formulae in Corollary~\ref{cor:main3}, it is clear that $\Psi_{T}$ is compatible with wedge products.

Let $Y_{1},Y_{2}\in\Tpoly^{\bullet}(\cN)$ be homogeneous cocycles.
Since $\alpha$ is a surjective quasi-isomorphism, $\ker \alpha$ is acyclic, and the cocycles $\beta(Y_{1})$ and $\beta(Y_{2})$ admit cocycle lifts $X_{1},X_{2}\in\Tpoly^{\bullet}(\cM)$ such that
\[ \alpha(X_{i})=\beta(Y_{i}),\qquad i=1,2\, . \]
Thus, $X_i$ and $Y_i$ are $\Psi$-related. Since the Schouten--Nijenhuis bracket of $\Psi$-related polyvector fields is again $\Psi$-related, we have
\[ \alpha([X_{1},X_{2}]) = \beta([Y_{1},Y_{2}])\, . \]
Therefore, $\Psi_{T}$ preserves the Schouten--Nijenhuis bracket. 
Conjugating $\Psi_{T}$ with $\hkr$, the map $\Psi_D$ is an isomorphism of Gerstenhaber algebras.
This completes the proof.
\end{proof}

\section*{Acknowledgment}
The author is grateful to Zhuo Chen, Hsuan-Yi Liao, Maosong Xiang and Ping Xu for their helpful comments and interest in this work.

\appendix

\section{DG vector bundles}
\subsection{Homotopy invariance of graded vector bundles}
It is a classical theorem that (see, for instance, \cite[Theorem~3.4.7]{MR1249482}, \cite[Section~VII.7.18]{MR336651}) the isomorphism classes of vector bundles are homotopy invariant. 
\begin{theorem}\label{thm:ClassicalHI}
Let $f:M\times [0,1]\to N$ be a morphism of smooth manifolds, and let $E\to N$ be a vector bundle. Then $f_{0}^{\ast}E \cong f_{1}^{\ast}E$ as vector bundles over $M$, where $f_{t}(m):=f(m,t)$ for all $m\in M$.
\end{theorem}

In this section, we prove the graded vector bundle analogue of the above theorem in the category of graded manifolds. As throughout this paper, all graded manifolds are finite-dimensional. However, we do not restrict ourselves to positive amplitude in this section; we allow either positive amplitude or negative amplitude. 
By definition, graded vector bundles are of finite rank: if $E=\bigoplus_{k}E^{k}$ is a graded vector bundle, then $E^{k}=0$ for all but finitely many $k$, and each $E^{k}$ is of finite rank.

We adopt the following notations: $M$ is a smooth manifold, and $I=[0,1]$ is the unit interval, often viewed as a graded manifold with trivial grading. Let $\cM$ be a graded manifold. For an open subset $U\subset M$, we denote by $\cM|_{U}$ the graded manifold obtained by restricting the base manifold of $\cM$ to $U$. 
We define $\cM_{I}=\cM\times I$, the product of graded manifolds $\cM$ and $I$, with base $M\times I$. The natural projections $\pr_{1}:\cM_{I}\to \cM$ and $\pr_{2}:\cM_{I}\to I$ are morphisms of graded manifolds.

Let $\cE_{I}$ be a graded vector bundle over $\cM_{I}$. For each $t\in I$, let $\iota_{t}:\cM\cong \cM\times \{t\} \hookrightarrow \cM_{I}$ denote the natural inclusion. We define $\cE_{t}:=\iota_{t}^{\ast}\cE_{I}$, which is a graded vector bundle over $\cM$.

Fix a $\tangent{\cM_{I}}$-connection $\nabla$ on $\cE_{I}$.
Since $\tangent{\cM_{I}}\cong \pr_{1}^{\ast}\tangent{\cM} \oplus \pr_{2}^{\ast}\tangent{I}$, the nowhere vanishing vector field $\partial_{t}\in \XX(I)$ can be viewed as an element of $\XX(\cM_{I})$.

\begin{lemma}\label{lem:A2}
Let $\cM$ be a graded manifold of either positive or negative amplitude.
For any $e_{0}\in \sections{\cM; \cE_{0}}$, there exists a unique element $e_{I}\in \sections{\cM_{I}; \cE_{I}}$ such that $e_{I}|_{\cM\times\{0\}}=e_{0}$ and $\nabla_{\partial_{t}}e_{I}=0$.
\end{lemma}

\begin{proof}
For each $p\in M$, there exists an open neighbourhood $U$ such that $\cM|_{U}\cong U\times V$ for some graded vector space $V$. 
After shrinking $U$, compactness of $I$ gives a subdivision $0=t_0<\cdots<t_r=1$ such that $\cE_I$ trivialises over $(\cM|_U)\times[t_{j-1},t_j]$ for each $j$.
It therefore suffices to solve the equation successively on these intervals; 
uniqueness ensures that the resulting local solutions agree on overlaps.
Without loss of generality, we may assume that $\cE_{I}$ is trivial over $U\times I$, i.e., 
\[\cE_{I}|_{\cM_{I}|_{U\times I}}\cong \cM_{I}|_{U\times I}\times W \cong ((U\times I)\times V) \times W\]
for some graded vector space $W$. Note that $W$ is a fibre of $\cE_{I}$.

It suffices to prove under the assumption that $\cM=M\times V$ and $\cE_{I}=((M\times I)\times V) \times W$.
Indeed, if the lemma holds for each open neighbourhood, then by uniqueness, the local solutions $e_{I}$ defined on overlapping neighbourhoods agree on their intersections. By gluing these local solutions, we obtain the global solution $e_{I}$.

Choose a homogeneous basis $\{w_{i}\}$ of $W$. Note that $\{w_{i}\}$ can be viewed as a local frame for both $\cE_{0}$ and $\cE_{I}$. Then it suffices to prove that, for each $w_{i}\in \sections{\cM; \cE_{0}}$, there exists a unique section
\[s_{i}(t)=\sum_{j} s_{i}^{j}(t)\cdot w_{j}\in \sections{\cM_{I}; \cE_{I}}\cong \smooth{\cM_{I}}\tensor W\]
such that $s_{i}^{j}(0)=\delta_{i}^{j}$ and $\nabla_{\partial_{t}}s_{i}(t)=0$, where $\delta_{i}^{j}$ is the Kronecker delta. If $e_{0}=\sum_{i}g^{i}\cdot w_{i}$ for some $g^{i}\in \smooth{\cM}$, then the unique solution is given by 
\[e_{I}=\sum_{i}g^{i}\cdot s_{i}(t)\, .\]

Let $\nabla_{\partial_t}w_k=\sum_jA_k^j(t)w_j$ and write
$s_i(t)=\sum_js_i^j(t)w_j$. Then
\begin{equation}\label{eq:AppODE}
\nabla_{\partial_t}s_i(t) = \sum_j\left( \frac{\partial s_i^j}{\partial t} +\sum_k s_i^kA_k^j \right)w_j =0 \, .
\end{equation}
Since $\nabla_{\partial_t}$ has degree $0$, we have
\[ s_i^j(t)\in \smooth{M\times I}\tensor (S(V^\vee))^{\degree{w_i}-\degree{w_j}}. \]
Because $V^\vee$ is concentrated in strictly positive or strictly negative degrees, each homogeneous component $(S(V^\vee))^d$ is finite-dimensional. 
Since $W$ has finite rank, the preceding equation therefore gives a finite system of linear ordinary differential equations with smooth parameters in $M$. 
The standard existence, uniqueness, and smooth-dependence theorem for ordinary differential equations gives the required solution.
This completes the proof.
\end{proof}

\begin{corollary}\label{cor:App12}
Let $\cM$ be a graded manifold of either positive or negative amplitude.
Fix $t\in I$. The assignment $e_{0}\mapsto e_{I}|_{\cM\times\{t\}}$ induces an isomorphism of $\smooth{\cM}$-modules
\[P_{t}:\sections{\cM; \cE_{0}} \cong \sections{\cM; \cE_{t}}\, .\]
In particular, $\cE_{0}\cong \cE_{1}$ as graded vector bundles over $\cM$.
\end{corollary}

We obtain the analogue of Theorem~\ref{thm:ClassicalHI} in the category of graded manifolds of either positive or negative amplitude.
\begin{theorem}\label{thm:AppMain}
For a graded manifold $\cM$ of either positive or negative amplitude, let $f:\cM\times [0,1] \to \cN$ be a morphism of graded manifolds. For any graded vector bundle $\cF\to \cN$, there is an isomorphism of graded vector bundles 
\[f_{0}^{\ast}\cF \cong f_{1}^{\ast}\cF\]
over $\cM\cong \cM\times \{0\} \cong \cM\times \{1\}$ where $f_{t}:=f|_{\cM\times \{t\}}$.
\end{theorem}
\begin{proof}
Apply Corollary~\ref{cor:App12} to $\cE_{I}:=f^{\ast}\cF$.
\end{proof}

\begin{remark}
For a graded manifold having both positive and negative degrees, the homogeneous components of $S(V^\vee)$ need not be finite-dimensional,
and the solution of Eq.~\eqref{eq:AppODE} need not be polynomial in the graded coordinates. 
If instead the structure sheaf is defined using formal power series,
\[\widetilde{\cA}|_{U}\cong \prod_{k \geq 0}\big( \cO_{M}|_{U}\tensor_{\RR} S^{k}(V^{\vee}) \big) \, ,\]
and all connections and bundle maps are continuous in the adic topology, then Lemma~\ref{lem:A2}, Corollary~\ref{cor:App12}, and Theorem~\ref{thm:AppMain} continue to hold.
\end{remark}

\subsection{DG vector bundles over DG manifolds of positive amplitude}
We now turn our attention to a DG manifold of positive amplitude $(\cM,Q)$ with base $M$.
Throughout this section, the pair $(\cE,\cQ_{\cE})$ denotes a DG vector bundle over $(\cM,Q)$ and its DG module of sections is denoted by $(\sections{\cE},\cQ_{\cE})$. Also, we use the notations $\fM=\smooth{M}$ and $\fcM=\smooth{\cM}$.

The following theorem describes the structures of DG vector bundles over $(\cM,Q)$.

\begin{theorem}\label{thm:BaseStr}
Let $(\cM,Q)$ be a DG manifold of positive amplitude with base $M$.
Given a DG vector bundle $(\cE, \cQ_{\cE})$, 
there exists a graded vector bundle $E=\bigoplus_{k=a}^{b}E^{k}$ over $M$ 
and an isomorphism of DG $(\fcM,Q)$-modules
\begin{equation}
(\sections{\cE},\cQ_{\cE})\cong (\fcM\tensor_{\fM}\sections{E},D)
\end{equation}
where the differential $D:\fcM \tensor_{\fM}\sections{E}\to \fcM \tensor_{\fM}\sections{E}$
can be written as 
$D=(d_{ij})_{a\leq i,j\leq b}$ 
with $d_{ij}:\fcM \tensor_{\fM}\sections{E^{j}}\to \fcM\tensor_{\fM}\sections{E^{i}}$
satisfying
\begin{enumerate}
\item $d_{ij}=0$ if $i<j$,
\item $d_{ii}=Q\tensor \id$ for all $i$,
\item $d_{ij}$ is $\fcM$-linear if $i>j$.
\end{enumerate}
\end{theorem}

\begin{proof}
By Proposition~\ref{prop:CatBundle}, we may identify $\cM=L$, where $\pi:L\to M$ is a graded vector bundle. Denote a point of $L$ by a pair $(p,v)$ where $p\in M$ and $v$ is a vector in the fibre over $p$.
For $t\in I$, define $f_{t}:L\to L$ by $f_{t}(p,v)=(p,tv)$. Then, $f_{1}=\id_{\cM}$ is the identity map on $\cM$ and $f_{0}=z\circ \pi$, where $z:M\to L$ is the zero section.

Let $E=\cE|_{M}$ be the restriction of $\cE$ to the base manifold $M\subset \cM$. By Theorem~\ref{thm:AppMain}, we have a bundle isomorphism $\cE\cong \pi^{\ast}E$. In terms of sections, we obtain an isomorphism of graded $\fcM$-modules
\[\sections{\cM;\cE} \cong \fcM \tensor_{\fM}\sections{M;E}\, .\]

To characterise the differential $D$, observe that the grading of $\fcM$ is concentrated in non-positive degrees, and that $Q(f)=0$ for all $f\in \fM$. 
Thus, the degree $1$ operator $Q\tensor \id : \fcM\tensor_{\fM}\sections{E}\to \fcM\tensor_{\fM}\sections{E}$ defines a well-defined $(\fcM,Q)$-module structure on $\fcM \tensor_{\fM}\sections{E}$. 
Moreover, the operator $D-(Q\tensor \id)$ is $\fcM$-linear, 
so each component of $D-(Q\tensor\id)$ is completely determined by its restriction to $\sections{E^j}$.

By degree considerations, one checks that $d_{ij}=0$ for $i<j$ and $d_{ii}-(Q\tensor \id) =0$ for all $i$. This completes the proof.
\end{proof}

We present several useful results concerning DG vector bundles over a DG manifold of positive amplitude.

\begin{lemma}\label{lem:KeyLem}
Let $(\cE,\cQ_{\cE})$ be a DG vector bundle over a DG manifold of positive amplitude $(\cM,Q)$, and let $(C,d_{C})$ be an acyclic DG $(\fcM,Q)$-module. Then the cochain complex $(C\tensor_{\fcM}\sections{\cE}, d_{C}\tensor \id + \id \tensor \cQ_{\cE})$ obtained by tensoring over $(\fcM,Q)$ is also acyclic.
\end{lemma}

\begin{proof}
We use the notations from Theorem~\ref{thm:BaseStr}.
Consider a descending filtration of graded vector bundles 
\[
\cE=F^{a}\cE \supset F^{a+1}\cE \supset \cdots \supset F^{b}\cE \supset F^{b+1}\cE=0 \, ,
\]
where each $F^{p}\cE$ is defined by
\[\sections{F^{p}\cE} \cong \bigoplus_{k=p}^{b}\fcM \tensor_{\fM} \sections{E^{k}}\, .\] 
By Theorem~\ref{thm:BaseStr}, 
each $\sections{F^{p}\cE}$ is closed under $\cQ_{\cE}$, thus one obtains a filtration of DG $(\fcM,Q)$-modules
\[(\sections{\cE}, \cQ_{\cE})=(\sections{F^{a}\cE}, D^{a}) \supset (\sections{F^{a+1}\cE}, D^{a+1}) \supset \cdots \supset (\sections{F^{b}\cE}, D^{b}) \supset (\sections{F^{b+1}\cE}, D^{b+1})=0 \, ,\]
where $D^{p}$ is the induced differential on $\sections{F^{p}\cE}$.
Note that, under the decomposition 
\[ \sections{F^{p}\cE} \cong \big( \fcM\tensor_{\fM}\sections{E^{p}} \big) \oplus \sections{F^{p+1}\cE}\, ,\]
the differential $D^{p}$ decomposes as
\[
D^{p}=
\begin{bmatrix} 
Q\tensor \id & 0\\
\sum\limits_{i=p+1}^{b}d_{ip} & D^{p+1} 
\end{bmatrix} \, .
\]
Then the cochain complex $(C\tensor_{\fcM} \sections{F^{p}\cE}, d_{C}\tensor \id + \id \tensor D^{p})$ 
obtained by tensoring $(C,d_{C})$ over $(\fcM,Q)$
is the total complex of the following double complex
\begin{equation}\label{eq:KeyProp1}
\begin{tikzcd}
\cdots\arrow{r} &  \big(C\tensor_{\fcM}\sections{F^{p+1}\cE}\big)^{q+p} \arrow{rrr}{d_{C}\tensor \id + \id \tensor D^{p+1}} &&&  \big(C\tensor_{\fcM}\sections{F^{p+1}\cE}\big)^{q+p+1} \arrow{r} & \cdots \\
\cdots \arrow{r} & \arrow{u}{\sum\limits_{i=p+1}^{b}d_{ip}} C^{q-1}\tensor_{\fM} \sections{E^{p}}  \arrow{rrr}{d_{C}\tensor \id} &&& \arrow{u}{\sum\limits_{i=p+1}^{b}d_{ip}}   C^{q} \tensor_{\fM} \sections{E^{p}} \arrow{r}& \cdots 
\end{tikzcd}
\end{equation}
where $C^{q}$ denotes the degree $q$ homogeneous component of $C$, and similarly, $(C\tensor_{\fcM}\sections{F^{p+1}\cE})^{q+p}$ denotes the degree $q+p$ homogeneous component of $C\tensor_{\fcM}\sections{F^{p+1}\cE}$.

Observe that $C\tensor_{\fcM} \sections{F^{b+1}\cE}=0$ is acyclic and that
\[(C\tensor_{\fcM}\sections{\cE}, d_{C}\tensor \id + \id \tensor \cQ_{\cE}) = (C\tensor_{\fcM} \sections{F^{a}\cE}, d_{C}\tensor \id + \id \tensor D^{a})\, .\]
We use the induction argument on $p$, starting from $p=b+1$, to prove the lemma:
if the cochain complex $(C\tensor_{\fcM} \sections{F^{p+1}\cE}, d_{C}\tensor \id + \id \tensor D^{p+1})$ is acyclic, then so is
$(C\tensor_{\fcM} \sections{F^{p}\cE}, d_{C}\tensor \id + \id \tensor D^{p})$.

Suppose that $(C\tensor_{\fcM} \sections{F^{p+1}\cE}, d_{C}\tensor \id + \id \tensor D^{p+1})$ is acyclic. Then the top row of \eqref{eq:KeyProp1} is exact. Now, since the DG $\fM$-module $(C,d_{C})$ is acyclic and the $\fM$-module $\sections{E^{p}}$ is projective, the complex $(C\tensor_{\fM} \sections{E^{p}}, d_{C}\tensor \id)$ is acyclic. That is, the bottom row of \eqref{eq:KeyProp1} is exact. 
By Acyclic Assembly Lemma~\cite[Lemma~2.7.3]{MR1269324}, the total complex of the double complex~\eqref{eq:KeyProp1} is acyclic. In other words, $(C\tensor_{\fcM} \sections{F^{p}\cE},d_{C}\tensor \id + \id \tensor D^{p})$ is acyclic. This completes the proof.
\end{proof}

\begin{corollary}\label{cor:KeyLem2}
Let $(\cE,\cQ_{\cE})$ be a DG vector bundle over a DG manifold of positive amplitude $(\cM,Q)$. Given a quasi-isomorphism of DG $(\fcM,Q)$-modules $\eta:(A,d_{A})\to (B,d_{B})$, the induced map
\[
\eta\tensor\id: (A\tensor_{\fcM}\sections{\cE}, d_{A}\tensor \id + \id \tensor \cQ_{\cE}) \to (B\tensor_{\fcM}\sections{\cE}, d_{B}\tensor \id + \id \tensor \cQ_{\cE})
\]
is also a quasi-isomorphism.
\end{corollary}
\begin{proof}
By~\cite[Corollary~1.5.4]{MR1269324}, a morphism of cochain complexes is a quasi-isomorphism if and only if its mapping cone is acyclic. Let $(C,d_{C})$ be the mapping cone of $\eta$. 
Then, the cochain complex $(C\tensor_{\fcM}\sections{\cE}, d_{C}\tensor \id + \id \tensor \cQ_{\cE})$
is the mapping cone of $\eta \tensor \id$, which, by Lemma~\ref{lem:KeyLem}, is acyclic. 
This completes the proof.
\end{proof}

\begin{corollary}\label{cor:KeyLem3}
Let $(\cE,\cQ_{\cE})$ be a DG vector bundle over a DG manifold of positive amplitude $(\cM,Q)$.
Suppose that the DG algebra $(\fcM,Q)$ of smooth functions on $(\cM,Q)$ is acyclic. Then the DG $(\fcM,Q)$-module of sections $(\sections{\cE},\cQ_{\cE})$ is also acyclic.
\end{corollary}
\begin{proof}
Apply Lemma~\ref{lem:KeyLem} with $(C,d_{C})=(\fcM,Q)$.
\end{proof}

\printbibliography
\end{document}